\documentclass[EJP,preprint]{ejpecp}
\usepackage[T1]{fontenc}
\usepackage{booktabs,array}
\graphicspath{{figures/}}
\SHORTTITLE{Persistent occupation and critical marked sums}
\TITLE{Persistent occupation and critical marked sums on branching random walks}
\AUTHORS{%
  Abderrahman~Skiredj\footnote{College of Computing, Mohammed VI Polytechnic
    University, Benguerir, Morocco; Ministry of AI, Digital Transition and
    Administrative Reform, Morocco. \BEMAIL{abderrahman.skiredj@um6p.ma}}
  \and
  Ismail~Berrada\footnote{College of Computing, Mohammed VI Polytechnic
    University, Benguerir, Morocco. \BEMAIL{ismail.berrada@um6p.ma}}}
\KEYWORDS{branching random walk; persistent local time; discounted tree sum;
critical moment; stopping line; ladder height}
\AMSSUBJ{60J80}
\AMSSUBJSECONDARY{60G50; 60G70}
\ABSTRACT{We study two path functionals of a branching random walk at an
exponential normalization. For centered integer displacements bounded above
by one, we construct on survival one ray whose local time at every large
spatial level is bounded below by a positive multiple of the square of that
level. For displacements in $\{-1,0,1\}$, overlapping forward profiles
give a single ray with lower coefficient $4t/(\sigma^2j_{1,1}^2)$, where
$\sigma^2$ is the tilted variance and $j_{1,1}$ the first positive zero
of $J_1$. The profile energy quotient is optimized exactly. For discounted marked
path sums with real displacements, we prove critical-moment sufficiency
without a drift assumption, and necessity under a finite ascending-ladder
mean and a weighted sibling moment. At positive tilted drift, a stopping-line
argument replaces the sibling condition by an $L\log L$ hypothesis. A sparse
mark construction identifies an exact universality obstruction in the
constant-displacement subclass. At negative tilted drift, weighted powers
give explicit sufficient logarithmic-tail conditions for every exponent
$t>0$ and examples excluding a universal critical-moment converse. Under
our sufficient conditions, path sums converge uniformly and attain their
maxima. The exact optimal
persistent local-time constant and a general negative-drift converse remain
open.}
\newcommand{\E}{\mathbb E}
\newcommand{\PP}{\mathbb P}
\newcommand{\T}{\mathcal T}
\newcommand{\ind}{\mathbf1}
\newcommand{\Y}{\mathcal Y}
\newcommand{\Ssurv}{\mathcal S}
\begin{document}
\section{Introduction}
A branching random walk assigns spatial displacements to a rooted family
tree. A vertex's position is the sum along its ancestral path. An
infinite \emph{ray} chooses one child at each generation. Figure~\ref{fig:overview}
shows the common setting of the two path questions studied here.

\clearpage
\begin{figure}[!t]
\centering
\includegraphics[width=\textwidth]{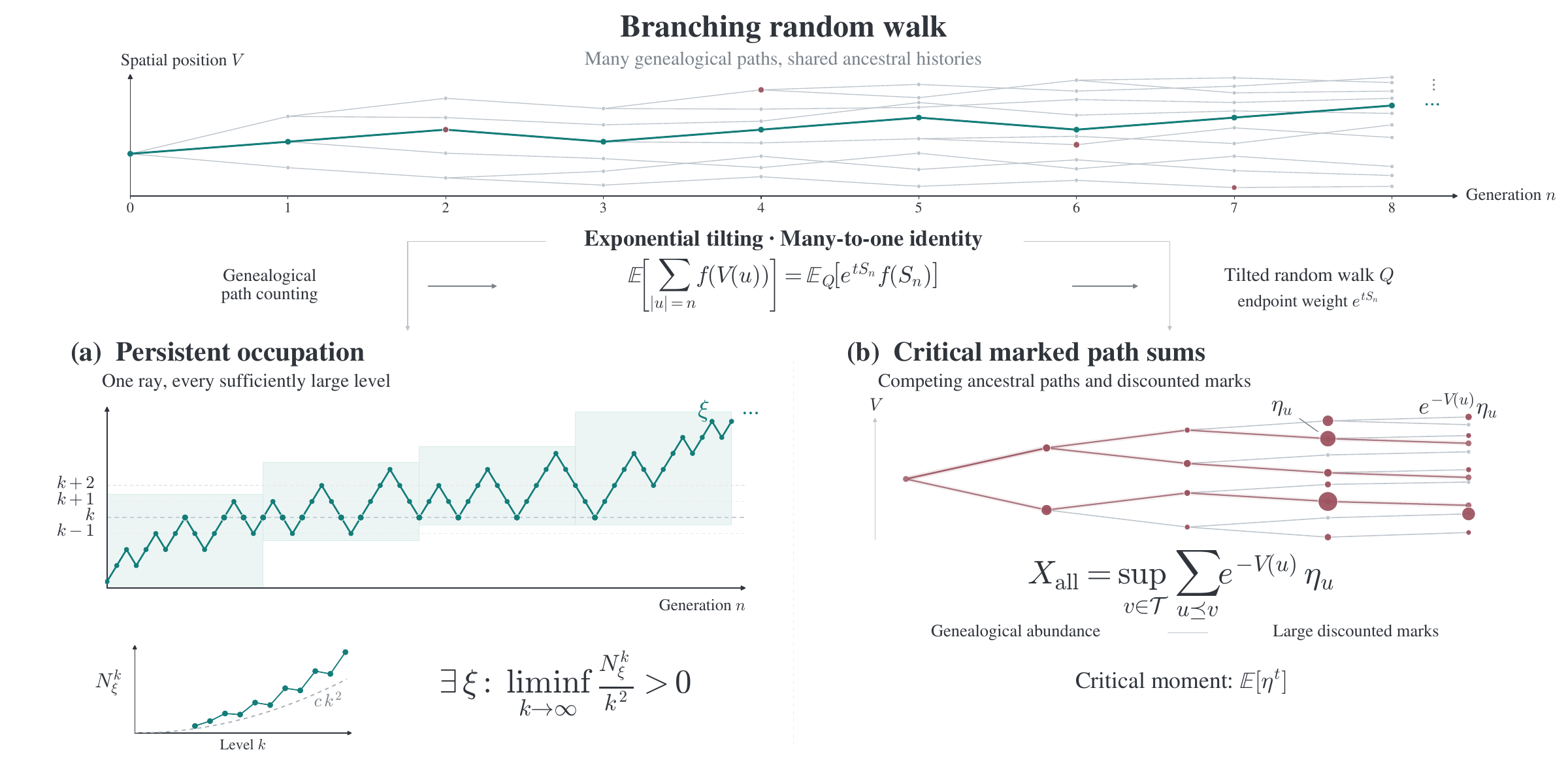}
\caption{The common branching random walk and its two path questions.
The upper panel sketches shared genealogical histories and the tilted-walk
counting identity. Panel (a) concerns occupation of every sufficiently
large level by one ray; panel (b) concerns discounted marks summed along
ancestral paths. The paths and profiles are schematic; the precise
hypotheses and conclusions are stated in Section~\ref{model}.}
\label{fig:overview}
\end{figure}

Two descendants may share a long ancestral prefix before their lineages
separate. This dependence matters in both questions.

The first question concerns occupation of spatial levels. For a ray
$\xi$, its local time at an integer level $k$ is the number of vertices
of $\xi$ positioned at $k$. We ask whether one ray can visit \emph{every}
sufficiently large level at least on the quadratic scale $k^2$.
Choosing a different successful ray for each level would give a weaker
statement. The second question attaches a nonnegative mark $\eta_u$ to
each vertex and discounts it by its position: its contribution is
$e^{-V(u)}\eta_u$. We ask when sums of these contributions remain bounded
over all paths, especially at the boundary between finite and infinite
mark moments. Both problems call for estimates that remain valid when
path segments are joined along the same genealogy.

The common device is an exponential change of measure. For a parameter
$t>0$, normalization by $\E\sum_i e^{-tY_i}=1$ turns the weighted child
displacements into the step law of one ordinary random walk. This
\emph{tilted walk} describes genealogical counting after an exponential
factor has been removed; it is generally different from choosing a child
uniformly. The factor rewards high endpoints, while the walk probability
penalizes rare path behavior. The results below make this competition
precise under their respective assumptions.

\subsection{Two calculations behind the questions}
Here is the change of measure in mathematical form. Write $N$ for the
number of children at one vertex and $Y_1,\ldots,Y_N$ for their
displacements. Under the normalization above,
\[
 Q(dy)=\E\sum_{i=1}^N e^{-tY_i}\delta_{Y_i}(dy)
\]
is a probability law. Let $S_0=0$ and let $S$ have independent increments
with this law. If $u_0,\ldots,u_n=u$ is the ancestral line of a vertex
at generation $n$, then independence of the reproduction lists gives,
for every nonnegative path functional $F$,
\[
 \E\sum_{|u|=n}F(V(u_0),\ldots,V(u_n))
 =\E_Q\!\left[e^{tS_n}F(S_0,\ldots,S_n)\right].
\]
The proof in Section~\ref{general-proof} recalls this many-to-one
identity. For example, if a path test $E_K$ ends at position $K$ and
$Z_K$ counts the generation-$n$ descendants passing it, then
$\E Z_K=e^{tK}\PP_Q(E_K)$. A rare event for one walk can therefore have
many witnesses in the tree. To find one \emph{persistent} witness, the
proof must also control how different witnesses share ancestors.

For occupation, let $\partial_\infty\T$ denote the infinite rays. Put
$N_\xi^k=\sum_{n\ge0}\ind_{\{V(\xi_n)=k\}}$. The order of
quantifiers is the point:
\[
 \sup_{\xi\in\partial_\infty\T}\liminf_{k\to\infty}
       \frac{N_\xi^k}{k^2}
 \ \le\ \liminf_{k\to\infty}
       \sup_{\xi\in\partial_\infty\T}\frac{N_\xi^k}{k^2}.
\]
The right side permits a new ray at each level; the left side requires
one ray to work through all sufficiently large levels. Theorem~\ref{A3:general}
concerns the left side. Its nested block construction is needed precisely
because a levelwise first-moment or second-moment estimate addresses
only the easier side of this inequality.

For marked sums, a deterministic-displacement example isolates the
critical exponent. Suppose $m=\E N\in(1,\infty)$, every child moves by
$c=(\log m)/t$, and the nonnegative marks are independent of the tree.
Then $me^{-tc}=1$ and a vertex in generation $n$ has discount
$e^{-V(u)}=m^{-n/t}$. Thus, for $d>0$,
\[
 \E\#\{u:|u|=n,\ e^{-V(u)}\eta_u>d\}
   =m^n\PP\{\eta>d m^{n/t}\}.
\]
If $\PP(\eta>x)$ has power order $x^{-s}$, this expectation has
exponential order $m^{n(1-s/t)}$. The exponential margin vanishes at
$s=t$; logarithmic tail corrections and reproduction properties then
matter. Equivalently, the critical moment is governed by
\[
 \E\eta^t=t\int_0^\infty x^{t-1}\PP(\eta>x)\,dx.
\]
These calculations identify the scale of the questions. They do not
assert boundedness or divergence of the full path sums, which require
the conditions and proofs below.

\subsection{Relation to earlier results and contributions}

A\"id\'ekon, Hu and Shi \cite{AHS} prove a local-time upper bound and ask
in Open Question~1(ii) of arXiv:2409.01048v2 whether the single-ray
quadratic local-time liminf can be positive, and for its optimal value.
We prove positivity under the skip-free hypotheses below; the optimum
remains open. Finite occupation blocks and an embedded surviving tree give
the ray. In the nearest-neighbor case, moving forward corridors first yield
a cumulative sine-square profile. General profiles then improve the coefficient
from approximately $0.10642t$ to $4t/j_{1,1}^2\simeq0.27244t$, with one ray
meeting the latter bound. A prefix constraint on the change-of-measure
likelihood supplies the branching second moment for these general profiles.
The resulting variational problem has a Bessel-function ground state; its
analytic solution is classical, while the persistent-ray construction is
the contribution proposed here. The extension also permits zero-displacement
children, with the coefficient divided by the tilted variance.

Corridor and small-deviation constructions are established tools for
branching random walks; see Fang and Zeitouni \cite{FZ} and Jaffuel
\cite{Jaffuel}. Bounds on consistent displacement or survival below a
barrier do not by themselves provide simultaneous quadratic occupation
at every large level by one ray. The persistence of that occupation and
the cumulative limiting profile are the objects considered here. The
polynomial summability threshold below is a consequence linking the two
parts of the paper; the exponent three also appears in the earlier
barrier and summability literature \cite{AHS,Jaffuel}.

For marks, the distinction between a sufficient moment and a necessary
moment is essential. AHS Theorem~1.3 compares a mark-tail exponent with
the normalization exponent under a strict inequality in either direction;
its statement does not decide equality. In the no-leaf constant-discount model,
Athreya \cite[Propositions~2 and~4(b)]{Athreya} proves a strict power-tail
sufficient condition and, with a finite offspring second moment, a necessary
critical-order tail bound. His fixed-point uniqueness results address a
different question. These statements do not give the critical-moment
criterion or the offspring-$L\log L$ boundary below.
A finite critical moment yields uniform convergence for independent marks
with arbitrary real tilted increments. Local dependence is allowed under a
weighted reproduction moment. Necessity
follows under a finite ascending-ladder mean and a sibling moment. At
strictly positive integrable drift, additive-martingale mass on first
crossing lines gives a different proof under an $L\log L$ condition.
The first-crossing device itself also appears in the proof of AHS
Theorem~1.3(i); here its $L^1$ line mass is used at the critical mark moment.
Negative drift behaves differently: even marks with infinite critical
moment can have uniformly bounded path sums.

The many-to-one identity, ladder renewals, and additive-martingale
convergence used here are classical; see Shi \cite{Shi}, Lawler and Limic
\cite{LL}, Durrett \cite{Durrett}, and Lyons \cite{Lyons}. The contribution is
the pathwise constructions and critical criteria obtained from these tools,
not new versions of those classical inputs. Max-type recursions provide
another formulation of the marked problem \cite{AB,Biggins98,Athreya}; we obtain the minimal
nonnegative endogenous solution in the finite case and do not assert
uniqueness among all recursive solutions. Inhomogeneous smoothing
transforms \cite{AM} usually involve a sum of weighted descendants;
the maximum over discounted path sums here is a different recursion.
The $L\log L$ obstruction uses the classical branching-process
martingale dichotomy \cite{LPP}.

Zhang \cite{Zhang} studies generation-discounted branching random walks by
generic chaining. With his notation $H=1/t$, his Theorem~1(i) assumes
$\E N^q<\infty$ for some $q>\max\{t,1\}$ in its sufficient direction;
our nonnegative independent-mark sufficiency requires no such offspring
moment. His Theorem~1(ii) already gives $L\log L$ necessity for constant
generation discounting, in a no-leaf framework allowing signed marks. We
do not claim that special case as new. Our critical positive-drift argument
uses $L^1$ mass on spatial first-crossing lines for unbounded real jumps
and permits local mark--reproduction dependence. The converse construction below shows that,
in the constant-displacement subclass, failure of $L\log L$ permits a
sparse independent mark law with infinite critical moment and uniformly
convergent path sums, even with a vanishing critical power tail. Sparse
marks do not settle the corresponding
regularly varying tail question. Related self-similar tree questions are considered by
Curien, Hu and Qian \cite{CHQ}; the present results concern discrete marked
genealogical paths and do not identify those tree models with this one.

Section~\ref{model} states the results and gives a short guide to the
arguments; the proofs and supporting material follow.

\section{Model and main results}\label{model}

At each vertex use an independent copy of $(\eta,\Y)$, where
$0\le\eta<\infty$ and $\Y=(Y_1,\ldots,Y_N)$ is a finite list of real
displacements. Siblings may be dependent, and $\eta$ may depend on its own
offspring list. The position $V(u)$ is the sum of displacements from the
root to $u$, with $V(\varnothing)=0$; $|u|$ is its generation and
$u\preceq v$ includes both endpoints of an ancestral path. An infinite
ray chooses one child at each generation. Suppose $\E N>1$, possibly
infinite, and, for some $t>0$,
\begin{equation}\label{C6:norm}
 \E\sum_i e^{-tY_i}=1,\qquad
 Q(B)=\E\sum_i e^{-tY_i}\ind_{\{Y_i\in B\}}.
\end{equation}
Let $S$ be the random walk with increment law $Q$ and $S_0=0$.
We distinguish expectation for this walk by the context. Assume throughout
that $Q\ne\delta_0$. (Under $\E N>1$ and \eqref{C6:norm}, this is automatic.)
Write
\[
 A(u)=e^{-V(u)}\eta_u,\quad
 X_{\rm all}=\sup_{v\in\T}\sum_{u\preceq v}A(u),\quad
 X_\infty=\sup_{\xi\in\partial_\infty\T}\sum_{u\in\xi}A(u).
\]
An empty supremum is zero. The boundary for $X_{\rm all}$ includes paths
ending at leaves, as in A\"id\'ekon--Hu--Shi (AHS) \cite{AHS}.
Set
\begin{equation}\label{C6:offspring}
 w_i=e^{-tY_i}(1+Y_i^+),\qquad Z=\sum_i w_i,
 \qquad \Omega=\E\sum_{i\ne j}w_iw_j.
\end{equation}
The sum in $\Omega$ is over ordered distinct siblings.

The local normalization weight is also denoted by
\[
 W=\sum_i e^{-tY_i},\qquad \E W=1.
\]
Throughout, expectations of walk functionals refer to the tilted law $Q$;
expectations of tree functionals refer to the original reproduction law.
The local pair at a vertex is independent of all strict ancestral pairs,
even when its mark and its own offspring are dependent.

\subsection{Persistent occupation}

At each vertex take an iid reproduction list $(Y_i)_{1\le i\le N}$, with integer $Y_i\le1$. Assume that, for some $t>0$,
\begin{equation}\label{A3:H}
 \E\sum_i e^{-tY_i}=1,\qquad
 \E\sum_iY_i e^{-tY_i}=0,\qquad
 0<\sigma^2:=\E\sum_iY_i^2 e^{-tY_i}<\infty,
\end{equation}
and
\begin{equation}\label{A3:Psi}
 \Psi:=\E\sum_{i\ne j}e^{-tY_i-tY_j}<\infty.
\end{equation}
No independence among siblings is assumed. Let $V(u)$ be positions, $\mathcal S$ survival, and $\partial_\infty\T$ the infinite rays. Put
\[
 N_\xi^k=\sum_{n\ge0}\ind_{\{V(\xi_n)=k\}},\qquad
 L=\sup_{\xi\in\partial_\infty\T}\liminf_{k\to\infty}\frac{N_\xi^k}{k^2},
\]
with empty supremum zero. Finite leaf rays have normalized liminf zero, so including them does not change $L$ on survival.

The tilted step law is
\begin{equation}\label{A3:spine}
 \PP(S_1\in dz)=\E\sum_i e^{-tY_i}\delta_{Y_i}(dz).
\end{equation}
It has mean zero, variance $\sigma^2$, and steps at most one. Its $+1$ probability is positive. Moreover $\E N=\E e^{tS_1}\in(1,e^t]$, by strict Jensen, and $\E N(N-1)\le e^{2t}\Psi<\infty$. Thus the tree is supercritical and locally finite. Jensen also gives $\E\sum_i e^{-sY_i}\ge1$ for every $s>0$, allowing the value infinity. Consequently \eqref{A3:H}--\eqref{A3:Psi} are precisely the boundary, skip-free-upward, and sibling hypotheses required here from AHS Theorem 1.6(ii) \cite{AHS}.

Write $\sigma_1=\inf\{n\ge1:S_n=0\}$ and let $\theta>0$ be the constant in AHS equation (5.4):
\begin{equation}\label{A3:excursion}
 q(h):=\PP\left(\min_{0\le j\le\sigma_1}S_j\le-h\right)\sim\frac\theta h.
\end{equation}
Donsker's theorem is used only for the walk path, with exactly the mean and variance assumptions in \eqref{A3:H}; see \cite[Theorems 8.1.4 and 8.1.11]{Durrett}.

\begin{theorem}[Full-hypothesis positivity]\label{A3:general}
There is a universal numerical constant $\kappa>0$ such that, under \eqref{A3:H}--\eqref{A3:Psi}, almost surely on $\mathcal S$, some infinite ray satisfies
\[
 \liminf_{k\to\infty}\frac{N_\xi^k}{k^2}\ge\kappa\frac t\theta.
\]
In particular $0<L\le t/(2\theta)<\infty$ on survival. This gives the positivity assertion in AHS Open Question 1(ii) under its stated hypotheses.
\end{theorem}

\begin{proposition}[Determinism]\label{A3:det}
Under \eqref{A3:H}--\eqref{A3:Psi}, $L$ is a deterministic constant on survival. Neither this statement nor Theorems \ref{A4:main} and \ref{A5:main} assert that a ray attains the optimal value $L$.
\end{proposition}

\begin{theorem}[Excursion normalization]\label{A4:theta}
For a centered integer walk with increments at most one and variance $\sigma^2\in(0,\infty)$,
\[
 \theta=\frac{\sigma^2}{2}.
\]
Consequently the AHS upper bound is $L\le t/\sigma^2$. Together with Theorem~\ref{A3:general}, $2\kappa t/\sigma^2\le L\le t/\sigma^2$ for a universal $\kappa>0$. The identity is derived here from standard renewal structure; no novelty claim is made for it.
\end{theorem}

\begin{theorem}[A continuous family attained by one ray]\label{A4:main}
Suppose additionally that $Y_i\in\{-1,1\}$. For every fixed $B>0$, almost surely on survival, one infinite ray satisfies
\begin{equation}\label{A4:CB}
 \liminf_{k\to\infty}\frac{N_\xi^k}{k^2}\ge t\mathcal C(B),
 \qquad
 \mathcal C(B)=\frac{2B}{\pi^2}\int_1^{1+B}
       \frac{\sin^2(\pi(r-1)/B)}{r^3}\,dr.
\end{equation}
The function $\mathcal C$ has a positive maximum at a finite $B>0$. A single ray also attains the lower bound $t\max_{B>0}\mathcal C(B)$. In particular, one ray satisfies
\begin{equation}\label{A4:simple}
 \liminf_{k\to\infty}\frac{N_\xi^k}{k^2}\ge\frac{17t}{160}.
\end{equation}
This is attainment of a proved constructional lower bound, not attainment of the unknown optimal $L$.
\end{theorem}

A ray meeting a displayed lower bound has normalized liminf
\emph{at least} that bound; equality for that ray is not asserted.
For $t=\log(2+\sqrt3)$, the rational bound is approximately
$0.1399267765$. Numerical optimization suggests a coefficient near
$0.1064210$, but neither that decimal nor uniqueness of a maximizing
parameter is used in any proof.

\begin{theorem}[Optimized forward profiles on one ray]\label{A5:main}
Under \eqref{A3:H}--\eqref{A3:Psi}, with $Y_i\in\{-1,0,1\}$, almost
surely on survival there is one infinite ray satisfying
\begin{equation}\label{A5:mainbound}
 \liminf_{k\to\infty}\frac{N_\xi^k}{k^2}
 \ge\frac{4t}{\sigma^2j_{1,1}^2}>\frac{t}{4\sigma^2},
\end{equation}
where $j_{1,1}$ is the first positive zero of the Bessel function $J_1$.
The coefficient $4/j_{1,1}^2\simeq0.27244$ is the exact supremum
\[
 \sup_{B>0}\ \sup_{0\ne f\in H_0^1(0,B)}
 \frac{\int_0^B f(x)^2(1+x)^{-3}\,dx}{\int_0^B f'(x)^2\,dx}.
\]
In particular $4t/(\sigma^2j_{1,1}^2)\le L\le t/\sigma^2$. The variational optimum is
attained as a lower bound by the constructed ray; the global value of
$L$ and existence of a ray attaining $L$ remain open.
\end{theorem}

Theorem~\ref{A5:main} strengthens the sine-family bound without adding
offspring assumptions and also allows zero displacements. Its proof
in Section~\ref{profile-extension}
uses profiles with increasing support in the same nested construction,
rather than an inference from a sequence of different near-optimal rays.

\begin{corollary}[A polynomial occupation threshold]\label{occupation-threshold}
Under the centered skip-free assumptions, almost surely on survival the
ray from Theorem~\ref{A3:general} satisfies, simultaneously for all
$0<b\le3$,
\[
 \sum_{n\ge0}(1+|V(\xi_n)|)^{-b}=\infty.
\]
For $b>3$, the corresponding maximal path sum is finite under the
sufficient criterion of A\"id\'ekon, Hu and Shi \cite{AHS}.
\end{corollary}
\begin{proof}
On the constructed ray, $N_\xi^k\ge ck^2$ eventually for some $c>0$.
Summing over positive integer levels gives the divergent series
$c\sum_{k\ge K}k^2(1+k)^{-b}$ when $b\le3$. This same ray works for every
such exponent. The finite direction is the polynomial-weight consequence
of the sufficient tree-sum criterion in \cite{AHS}, under the present
boundary assumptions.
\end{proof}

\subsection{Critical marks and the assumptions for each direction}

\begin{theorem}[General critical-moment sufficiency]\label{C6:suff}
Under \eqref{C6:norm}, if $\E\eta^t<\infty$, then $X_{\rm all}<\infty$
almost surely in either of the following cases:
\begin{enumerate}
\item $\eta$ is independent of its offspring list;
\item $\E Z^p<\infty$ for some $p>1$.
\end{enumerate}
No drift assumption, lattice condition, or increment moment is needed.
Convergence of the path sums is uniform in genealogical depth. Their maximum
is attained on the boundary including finite leaf paths; on survival, the
maximum restricted to infinite rays is also attained. No assumption on
$\Omega$ is required for this theorem.
\end{theorem}

Let
\[
 T^+=\inf\{n\ge1:S_n>0\},\qquad H^+=S_{T^+}.
\]
Our ladder assumption is
\begin{equation}\label{C6:ladder}
 \PP(T^+<\infty)=1,\qquad 0<\E H^+<\infty.
\end{equation}

\begin{theorem}[Necessity and exact criterion]\label{C6:nec}
Under \eqref{C6:norm}, \eqref{C6:ladder}, and $\Omega<\infty$, if
$\E\eta^t=\infty$, then
\[
 \sup_{u\in\T}A(u)=X_{\rm all}=\infty
 \quad\hbox{almost surely on survival}.
\]
This direction allows arbitrary local mark--reproduction dependence and
does not require a moment of $Z$ beyond the sibling assumption.
Combining with Theorem~\ref{C6:suff} gives an exact criterion for independent
marks, and for dependent marks with some finite $\E Z^p$, $p>1$.

A sufficient increment condition for \eqref{C6:ladder} is
\begin{equation}\label{C6:onesided}
 \E|S_1|<\infty,\qquad \E S_1\ge0,
 \qquad \E(S_1^+)^2<\infty.
\end{equation}
In particular the result covers all nonnegative-drift finite-variance real
walks, and certain walks with infinite negative second moment.
\end{theorem}

Let $q$ be the extinction probability and $s=1-q$.
\begin{corollary}[Exact infinite-ray criterion]\label{C6:backbone}
Under \eqref{C6:norm}, \eqref{C6:ladder}, $\Omega<\infty$, and some
$\E Z^p<\infty$, $p>1$, finiteness of $X_\infty$ almost surely on
survival is equivalent to
\begin{equation}\label{C6:raymoment}
 \E[\eta^t\ind_{\{N\ge1\}}]<\infty.
\end{equation}
In the finite case convergence is uniform and the maximum is attained.
In the infinite case there is an infinite ray with unbounded individual
amplitudes $A(u)$. The divergence direction needs only $\Omega<\infty$.
For independent marks, $\Omega<\infty$ suffices for the equivalence, with
criterion $\E\eta^t<\infty$. The finite direction, with the appropriate
independence or $Z^p$ assumption, does not need \eqref{C6:ladder} or $\Omega$.
\end{corollary}

For the positive-drift result assume
\begin{equation}\label{C7:positive}
 \E|S_1|<\infty,\qquad \mu:=\E S_1>0.
\end{equation}
Write \(H^+=S_{T^+}>0\), where
\(T^+=\inf\{n\ge1:S_n>0\}\).
Both the properness and the finite mean of this height follow from
\eqref{C7:positive}; the short argument is included below. No second
increment moment or separate ladder-moment assumption is imposed.

\begin{theorem}[Necessity under a logarithmic moment]\label{C7:positive-main}
Under \eqref{C6:norm}, \eqref{C7:positive}, and
\begin{equation}\label{C7:llogl}
 \E[W\log^+W]<\infty,
\end{equation}
if \(\E\eta^t=\infty\), then
\[
 \sup_{u\in\T} A(u)=X_{\rm all}=\infty
 \quad\hbox{almost surely on }\Ssurv.
\]
The mark may depend on its own reproduction. No condition on \(\Omega\)
is required.

For independent marks these assumptions give the exact criterion
\(\E\eta^t<\infty\), with uniform convergence and attainment in the finite
case. For dependent marks the same equivalence holds if
\(\E Z^p<\infty\) for some \(p>1\); this latter assumption already implies
\eqref{C7:llogl}.

On infinite rays under dependence, the corresponding criterion is
\(\E[\eta^t\ind_{\{N\ge1\}}]<\infty\), with the same direction-specific
conditions. In its divergent case, one ray has unbounded amplitudes.
\end{theorem}

Finiteness and attainment use Theorem~\ref{C6:suff}; the positive-drift converse does not require \(\Omega\).
Under the finite first increment moment in \eqref{C7:positive},
\(\Omega<\infty\) implies \eqref{C7:llogl}; an example below shows the
implication is strict.

\begin{theorem}[Exact universality boundary in a subclass]\label{C7:universality}
Fix a Galton--Watson law with \(1<m:=\E N<\infty\), and a number \(t>0\).
Give every child the deterministic displacement \(c=(\log m)/t\).
Then the following are equivalent:
\begin{enumerate}
\item \(\E[N\log^+N]<\infty\).
\item For every nonnegative finite independent mark law, finiteness of
\(X_{\rm all}\) almost surely on survival is equivalent to
\(\E\eta^t<\infty\).
\end{enumerate}
The same equivalence holds with \(X_\infty\) in place of \(X_{\rm all}\).
If the logarithmic offspring moment is infinite, there exists an
independent mark law with
\[
 \E\eta^t=\infty,\qquad
 \E\eta^r<\infty\quad(0<r<t),
\]
for which the path sums converge uniformly and are bounded almost surely.
The maximum over finite and infinite paths is attained; the maximum over
infinite rays is attained on survival. The constructed mark law is sparse; no
regular variation of its tail is asserted.
\end{theorem}

The sparse law in the proof can moreover be chosen with
\(x^t\PP(\eta>x)\to0\). This is a tail-order statement, not a finite
critical moment: \(\E\eta^t=\infty\) still holds.

The \(L\log L\) necessity direction for constant generation discounting
is already present in Zhang \cite[Theorem 1(ii)]{Zhang}, under a no-leaf
assumption and in a framework also allowing signed increments. We do not
claim that special case as a new theorem. The results here concern spatial
stopping lines, the converse universality obstruction, and the
negative-drift estimates.

Here is a convenient consequence of the negative-drift theorem proved in
Section~\ref{negative-proof}.

\begin{corollary}[Negative drift at every exponent]\label{C7:negative-main}
Assume \eqref{C6:norm}, \(\E S_1^+<\infty\), and \(\E S_1<0\), allowing
the value \(-\infty\). Suppose
\[
 \PP(\eta>x)\le Cx^{-t}(1+\log x)^{-\beta}\quad(x\ge1)
\]
for some \(\beta\in\mathbb R\). If, for some \(s>1\),
\begin{equation}\label{C7:simple-negative}
 \E(S_1^+)^s<\infty,\qquad
 \beta>1+2(t-1)_+-s,
\end{equation}
then \(X_{\rm all}<\infty\), path sums converge uniformly, and the
maximum is attained. No \(\Omega\), \(Z^p\), or independence assumption
on the mark is needed.
\end{corollary}

For a weak critical tail \(Cx^{-t}\), this gives a positive increment
moment of order \(s>\max\{1,2t-1\}\) as a sufficient condition. These are
sufficient bounds, not claimed necessary or optimal. For \(0<t<1\),
Theorem~\ref{C6:neglog} also covers some endpoint cases not
included in the strict inequality \eqref{C7:simple-negative}.

The following table separates the hypotheses for the two
directions. All rows retain the common model and normalization assumptions.
\begin{center}\small\renewcommand{\arraystretch}{1.14}
\begin{tabular}{@{}p{.27\textwidth}p{.68\textwidth}@{}}
\toprule
Conclusion & Additional assumptions \\
\midrule
Uniform finite path sums & $\E\eta^t<\infty$; independent marks or
$\E Z^p<\infty$ for some $p>1$. No drift condition.\\
Ladder converse & Proper finite-mean strict ascending height;
$\Omega<\infty$; $\E\eta^t=\infty$. Local dependence allowed.\\
Positive-drift converse & $\E|S_1|<\infty$, $\E S_1>0$,
$\E[W\log^+W]<\infty$; $\E\eta^t=\infty$.\\
Infinite rays with dependence & Use
$\E[\eta^t\ind_{\{N\ge1\}}]$ with the assumptions for the relevant
finite or divergent direction.\\
Negative-drift sufficiency & The tail and positive-increment moment
conditions of Corollary~\ref{C7:negative-main} or
Theorem~\ref{C7:power}. No general converse is claimed.\\
\bottomrule
\end{tabular}\end{center}

\subsection{Guide to the proof mechanisms}\label{proof-guide}
We record the ideas that connect the two parts of the paper. The
statements above give the precise hypotheses; the arguments below are
only a guide to their proofs.

\paragraph{From the tree to one walk.}
The normalization in \eqref{C6:norm} makes $Q$ a probability law. At one
generation its definition gives
\[
 \E\sum_i F(Y_i)=\E_Q[e^{tS_1}F(S_1)]\qquad(F\ge0).
\]
Iteration yields the many-to-one identity \eqref{A3:m21}. Thus counting
tree paths ending near height $K$ produces an $e^{tK}$ factor, while the
remaining question is the probability of a suitable event for one walk.
This explains why a walk event that is exponentially rare in $K$ may
still occur among many descendants. The second moment must then account
for pairs of descendants sharing an ancestor; $\Psi$ in the occupation
problem and $\Omega$ in the marked ladder converse control the
corresponding weighted sibling terms.

\paragraph{Why occupation blocks can be joined.}
The target in Theorem~\ref{A3:general} concerns the same ray at every
large level. A separate successful path for each level would not suffice.
The construction therefore tests finite path segments, keeps many
successful endpoints, and starts the next test from those endpoints.
Their own reproduction has not been revealed, so the next tests use
fresh descendant trees. A corridor of width $K$ suggests the scale:
diffusive motion takes order $K^2$ steps to explore it, and order $K$
such excursions give order $K^2$ visits to an interior level. The
resulting block lasts order $K^3$ steps. Its walk cost competes with
the $e^{tK}$ genealogical factor. When the number of endpoints and their
success probabilities leave enough descendants for every later stage,
the retained tree contains one infinite ray. This is the role of the
burst and survival lemmas, rather than an inference from unrelated
levelwise witnesses.

\paragraph{Where the Bessel coefficient comes from.}
For the quantitative result, a positive profile $f$ favors some sites
within a forward corridor. The transformed walk spends a fraction
proportional to $f^2$ at each site and pays a cost governed by the
Dirichlet energy $\int f'^2$. Choosing a profile therefore amounts to
balancing occupation against that energy. Control of the likelihood at
every prefix is needed because two descendants may branch before the
end of a block; a terminal estimate alone would not bound both suffixes.
Over successive corridors, a level $k$ receives contributions from
earlier starting heights $H\simeq uk$. With $x=u^{-1}-1$, the factor
$u\,du$ becomes $(1+x)^{-3}\,dx$ in absolute value. This produces the
weighted quotient in Theorem~\ref{A5:main}. Its positive ground state
solves a one-dimensional differential equation with a Bessel solution.
Increasing the support of approximating profiles through the same
nested construction lets one ray meet the quotient's supremum; it does
not identify the global optimum over all rays.

\paragraph{Why the marked problem is critical.}
At height $kb$, the discount $e^{-V(u)}$ makes a mark comparable to one
when $\eta_u$ is of order $e^{kb}$. The many-to-one count contributes
roughly $e^{tkb}$, whereas a finite $t$th mark moment controls the
geometric tail sum
\[
 \sum_{k\ge0}e^{tkb}\PP\{\eta>d e^{kb}\}
 \le C_{t,b}d^{-t}\E\eta^t.
\]
At this exponent the exponential factors cancel, so a pointwise tail
estimate is too crude. Killed-walk potentials give spatial band counts;
the probability that a path contains two comparable large marks is then
summable, leaving uniformly small tails of the discounted path sums.
For the converse, ladder renewal bounds supply enough visits to high
bands, and a second-moment argument turns an infinite critical moment
into infinitely many large amplitudes. Under positive tilted drift the
first-crossing line and its martingale mass give a different route;
under negative drift, weighted powers give sufficient bounds without a
general converse. The assumptions for these directions are deliberately
listed separately in the table above.

\section{Finite occupation blocks and a persistent ray}\label{general-proof}

The many-to-one identity is
\begin{equation}\label{A3:m21}
 \E\sum_{|u|=n}F(V(u_0),\ldots,V(u_n))
 =\E_0[e^{tS_n}F(S_0,\ldots,S_n)],\qquad F\ge0.
\end{equation}
It follows by induction from \eqref{A3:spine}. Summing first-stopping events over $n$ gives its stopped-line version by Tonelli. We allow a stopping rule to reject a path by assigning it time infinity.

\begin{lemma}[General endpoint second moment]\label{A3:coarse}
Fix $K\ge1$ and a finite integer $n_*$. Let a stopping rule accept only paths ending at $K$ by time $n_*$, whose prefixes stay in $[-dK,(1+d)K]$. Let $Z_K$ count accepted vertices. If the corresponding walk acceptance probability is $a_K$, then
\begin{equation}\label{A3:coarseZ}
 \E Z_K=e^{tK}a_K,\qquad
 \E Z_K^2\le e^{tK}+\Psi n_*e^{(2+d)tK}.
\end{equation}
\end{lemma}
\begin{proof}
The first moment is the stopped many-to-one formula. For the square, the diagonal is $\E Z_K\le e^{tK}$. For two distinct accepted endpoints, condition on their most recent common ancestor $w$, its full prefix, and its reproduction list. Given $V(w)=x$, the expected number accepted through child $i$ is at most $e^{t(K-x-Y_i)}$: the remaining conditional acceptance probability of the spine is at most one. Descendant trees are independent after the list is fixed. The ordered cross contribution is at most $\Psi e^{2t(K-x)}$ after averaging the list. Sum over common ancestors at generations $j<n_*$ using \eqref{A3:m21}; their prefixes stay in the corridor, so $e^{-tS_j}\le e^{tdK}$. This gives \eqref{A3:coarseZ}. Accepted vertices form an antichain, so no accepted endpoint can be an ancestor of another one.
\end{proof}

\begin{lemma}[Bursts give a nested ray]\label{A3:burst}
Suppose block tests at scales $K_j=2^jK_0$ start at relative zero, end at $K_j$, and form finite stopping lines with unused reproduction below their endpoints. Suppose their counts dominate a law with
\[
 R_j=\lfloor e^{g_jK_j}\rfloor\ \hbox{children with probability }P_j=e^{-h_jK_j},
\]
and zero children otherwise. At every stage used, require
\[
 1\le R_j<\infty,\qquad 0<P_j\le1.
\]
If
\begin{equation}\label{A3:gap}
 \sum_{j\ge1}\exp\{-P_jR_{j-1}\}<\infty,
\end{equation}
then the retained process has positive survival probability. On survival an infinite ray concatenates the certified block segments. A constant pair $g_j=g,h_j=h$ with $g>2h\ge0$ satisfies these conditions and \eqref{A3:gap}. If the construction starts only at a later stage, access to that stage with at least one parent must separately have positive probability.
\end{lemma}
\begin{proof}
Stopping-line exploration exposes no reproduction below an accepted endpoint. The endpoints are incomparable, so their future tests are independent. This is also immediate by revealing the finite tree to the maximal permitted generation and stopping exploration at each accepted endpoint. Thinning realizes the stated dominated law.

Require success of the first parent, an event of probability $P_0>0$.
After any stage-$j-1$ success there are at least $R_{j-1}\ge1$
parents for the next stage. Their conditional joint failure probability
is at most
\[
 (1-P_j)^{R_{j-1}}\le \exp(-P_jR_{j-1})=:a_j<1.
\]
By \eqref{A3:gap}, $\sum_j a_j<\infty$, so eventually $a_j\le1/2$
and $\log(1-a_j)\ge-2a_j$. The tail of the product is strictly
positive, and its finite initial product is positive because every
$a_j<1$. Thus a lower bound for indefinite continuation is
\[
 P_0\prod_{j\ge1}(1-e^{-P_jR_{j-1}})>0.
\]
The retained tree is finitely branching; K\"onig's lemma supplies an actual nested chain. This argument retains all successful children and does not concatenate unrelated witnesses.
\end{proof}

\begin{lemma}[Amplification on survival]\label{A3:amplify}
Any positive-probability relative subtree construction of a ray with a prescribed lower bound on $\liminf_k N_\xi^k/k^2$ gives such a ray almost surely on $\mathcal S$.
\end{lemma}
\begin{proof}
Spatial integer shifts and finite ancestral prefixes leave this liminf unchanged. At generation $n$, relative constructions in distinct subtrees are independent, each with the same success probability $\rho>0$. Let $Z_n$ be the original population and $F$ the event of no construction success anywhere. For every integer $B$,
\[
 \PP(F\cap\mathcal S)\le\PP(\mathcal S,Z_n<B)+(1-\rho)^B.
\]
On survival $Z_n\to\infty$. To recall an elementary proof, if $\PP(N=0)>0$, each visit to population sizes $1,\ldots,B$ has a uniformly positive conditional chance of extinction in the next step, so infinitely many such visits and survival have probability zero. If $\PP(N=0)=0$, populations are nondecreasing; at any fixed positive size, the probability of remaining there forever is zero since $\PP(N=1)<1$. Let $n$ and then $B$ tend to infinity. Conditioning independent subtrees on joint survival was not used.
\end{proof}

\section{A corridor block for the general skip-free spine}
Fix $d=1/16$ in this section and set
\[
 I_K=[-\lfloor dK\rfloor,K+\lfloor dK\rfloor]\cap\mathbb Z,
 \qquad J_K=[-dK/4,dK/4]\cap\mathbb Z.
\]
Let $\tau_I$ be exit from $I_K$.

\begin{lemma}[Two-sided excursion control]\label{A3:reverse}
The maximum of a first-return excursion has the same distribution as the negative of its minimum. In particular, for all sufficiently large $K$, uniformly in $z\in\{0,\ldots,K\}$, a fresh excursion from $z$ leaves $I_K$ with probability at most $8\theta/(dK)$.
\end{lemma}
\begin{proof}
A finite first-return increment sequence $(y_1,\ldots,y_n)$ has total sum zero and no earlier zero partial sum. Reversing its order preserves these properties and its probability, the product of the increment probabilities. Its new partial sums are $-S_{n-j}$, so its maximum is the negative of the original minimum. Recurrence makes the first-return time finite almost surely; sum over lengths and paths. Holding excursions of length one cause no exception.

From any target $z$, the two corridor boundaries are at distance at least $\lfloor dK\rfloor$. By the union bound, excursion reversal, and \eqref{A3:excursion}, the exit probability is at most $2q(\lfloor dK\rfloor)\le8\theta/(dK)$ for large $K$.
\end{proof}

\begin{lemma}[A bounded-time occupation trial]\label{A3:trial}
There exist universal constants $T<\infty$, $p\in(0,1/2)$, and
$\beta=d/(32\theta)>0$, such that for all large integers $K$, all starting sites $x\in J_K$, and every target $z\in\{0,\ldots,K\}$, the following event has probability at least $p$:
\begin{enumerate}
\item the walk stays in $I_K$ through $n_K=\lceil TK^2/\sigma^2\rceil$;
\item $S_{n_K}\in J_K$;
\item during times $0,\ldots,n_K-1$ it visits $z$ at least $m_K=\lfloor\beta K\rfloor$ times.
\end{enumerate}
There is also a universal $p_H>0$ such that, uniformly in $x\in J_K$, the walk hits $K$ within $b_K=\lceil K^2/\sigma^2\rceil$ steps before exiting $I_K$ with probability at least $p_H$. The sufficiently large threshold for $K$ may depend on the step law.
\end{lemma}
\begin{proof}
We give all ingredients, including their uniformity.

First, Donsker's theorem, in the normalization $S_{\lfloor sK^2/\sigma^2\rfloor}/K$, has standard Brownian motion as its limit. It supplies a universal positive lower bound $p_H$ for hitting any target $z\in[0,K]$ within $b_K$ steps before exit, uniformly over $x\in J_K$. To see why only path convergence is needed, use an open Brownian corridor event that visits a point below $-d/2$ near time $1/3$ and a point above $1+d/2$ near time $2/3$, staying in $(-d,1+d)$ throughout. A tube about a polygonal path realizes this event with positive probability, uniformly over starting points in the compact interval $[-d/4,d/4]$. Such tube positivity follows, for example, by dividing into finitely many intervals, requiring each Gaussian increment in an open interval and each Brownian bridge in a small tube. Donsker and a compact-subsequence argument transfer a fixed positive lower bound to the scaled walk. Any integer walk with upward steps at most one must hit every integer target while crossing upward between these two positions. Periodicity is irrelevant to this interval/path event. The same event proves the final assertion with target $K$.

Second, write $b_0=\PP(\mathcal N(0,1)>2)>0$. The central limit theorem gives $\PP(S_{b_K}>2K)\ge b_0/2$ for large $K$. Such a displacement forces exit from $I_K$ regardless of its starting point. The Markov property on successive blocks of $b_K$ steps gives, for large $K$,
\[
 \sup_{y\in I_K}\E_y\tau_I\le2b_K/b_0\le4K^2/(b_0\sigma^2).
\]
Set $T_1=16/b_0$. Starting at a target $z$, Lemma \ref{A3:reverse} and a union bound over $m_K-1$ excursions show that $m_K$ visits occur before exit with probability at least $3/4$. Meanwhile $\PP_z(\tau_I>T_1K^2/\sigma^2)\le1/4$. Therefore $m_K$ visits occur before exit and by time $T_1K^2/\sigma^2$ with probability at least $1/2$. At the time of the last required visit the walk is again exactly at $z$.

Finally choose $T=T_1+4$. From any target in $[0,K]$, and for any integer remaining duration whose rescaled value lies in $[1,T+1]$, Donsker supplies a universal uniform probability $p_R>0$ of staying in $I_K$ and ending in $J_K$. Indeed, for standard Brownian motion the probability of remaining in the open corridor and ending in $(-d/4,d/4)$ is positive, and has a positive infimum over the compact set of starting points $[0,1]$ and durations $[1,T+1]$. This can equally be proved by finitely many polygonal tubes. Failure of a uniform discrete bound would yield a convergent sequence of starting points and durations contradicting the invariance principle and the open-event lower bound.

First hit $z$, then collect its visits, then use the remaining time to return to $J_K$. The first two phases finish by $(1+T_1)K^2/\sigma^2+1$, leaving at least $K^2/\sigma^2$ for large $K$. Apply the strong Markov property at their stopping times. The probability is at least $p_Hp_R/2$; choose once and for all a universal $p<\min(1/2,p_Hp_R/2)$. The visit-collection phase ends strictly before $n_K$, so its visits are counted in the stated half-open time block.
\end{proof}

\begin{lemma}[Filling a whole interval]\label{A3:fill}
There is a positive constant $b$ and a stopping-line test ending at $K$ by $O(K^3)$ steps, staying in $I_K$, and certifying at least $bK^2$ visits to every integer in $[0,K]$, whose walk probability is at least
\begin{equation}\label{A3:fillprob}
 \frac{p_H}{2}e^{-tK/100}
\end{equation}
for all sufficiently large $K$.
\end{lemma}
\begin{proof}
Let $F_K$ be the corridor-and-endpoint event in the first two items of Lemma \ref{A3:trial}. Its probability from each $x\in J_K$ is at least $p$. For each block of length $n_K$, condition its law on $F_K$, starting from the current endpoint. This defines a Markov law at block boundaries, all of which lie in $J_K$. Under it the conditional probability that a fixed target receives $m_K$ visits in the next block is at least $p$: the original joint event has probability at least $p$, and division by $\PP_x(F_K)\le1$ cannot reduce it.

Put
\[
 a=\frac{t}{100(1+|\log p|)},\qquad B_K=\lfloor aK\rfloor.
\]
The conditional Bernoulli Chernoff bound gives probability at most $e^{-pB_K/8}$ that a fixed target has fewer than $pB_K/2$ successful blocks. It follows by iterating the conditional Laplace-transform inequality $\E[e^{-sI}\mid\mathcal F]\le1-p+pe^{-s}$. A union bound over $K+1$ targets tends to zero. For large $K$, with conditional probability at least $1/2$, every target receives at least
\[
 m_K\frac{pB_K}{2}\ge \frac{\beta pa}{8}K^2=:bK^2
\]
visits. Block time intervals are disjoint.

Returning to the original walk law costs the product of the block normalizers, at least $p^{B_K}\ge e^{-tK/100}$. After these blocks the position lies in $J_K$. Run until the first hit of $K$, rejecting on corridor exit or after an additional $b_K$ steps. Lemma \ref{A3:trial} supplies probability $p_H$. This is a stopping rule, with finite accepted time at most $B_Kn_K+b_K=O(K^3)$ for the fixed step law; it proves \eqref{A3:fillprob}.
\end{proof}

\begin{proof}[Proof of Theorem \ref{A3:general}]
For the test in Lemma \ref{A3:fill}, Lemma \ref{A3:coarse} gives
\[
 \E Z_K\ge(p_H/2)e^{99tK/100},\qquad
 \E Z_K^2\le e^{tK}+O(K^3)e^{(2+1/16)tK}.
\]
Paley--Zygmund implies, for all large $K$,
\[
 \PP(Z_K\ge\lfloor e^{3tK/4}\rfloor)\ge e^{-tK/8},
\]
because $t/16+2t/100<t/8$. At a selected vertex at absolute position $K_j$, the block of size $K_j$ ends at $2K_j=K_{j+1}$ and fills every level in $[K_j,2K_j]$. Lemma \ref{A3:burst} applies with $g=3t/4$ and $h=t/8$.

An initial chain of $K_0$ positive steps has positive probability: at each step choose the first $+1$ child, which exists with positive probability since the tilted $+1$ mass is positive. Its last endpoint's reproduction is unused. On survival of the retained process, for $K_j\le k<2K_j$, its one nested ray has $N_\xi^k/k^2\ge b/4$. Thus a valid deterministic lower bound is
\begin{equation}\label{A3:cgeneral}
 c=\frac{\beta pa}{32}=\kappa\frac t\theta,
 \qquad \kappa=\frac{p}{1638400(1+|\log p|)}>0,
\end{equation}
in terms of the constants explicitly defined in Lemmas \ref{A3:trial}--\ref{A3:fill}. Lemma \ref{A3:amplify} upgrades positive probability to almost surely on survival. AHS Theorem 1.6(ii) gives the upper bound $L\le t/(2\theta)$. No extra offspring or increment moment was introduced.
\end{proof}

\begin{proof}[Proof of Proposition \ref{A3:det}]
Work in the completed probability space. The relation between a realized labelled tree and a ray with a prescribed Borel liminf property is Borel; projecting over rays gives analytic, hence universally measurable, existence events. Thus $L$ is measurable in the completion. Equivalently one can use the countable rational threshold events $\{L>c\}$.

Finite prefixes and fixed integer spatial shifts leave the normalized liminf invariant. Every infinite ray enters one child subtree, so
\[
 L=\max_{1\le i\le N} L_i,
\]
where $L_i$ is the relative value in the $i$th subtree and extinct values are zero. Conditional on reproduction, these are iid copies of $L$. For any rational $c\ge0$, $r_c=\PP(L\le c)$ therefore satisfies $r_c=\E r_c^N$. The only fixed points of the supercritical offspring generating function in $[0,1]$ are its extinction probability $q$ and one. One elementary justification is strict convexity of $f(z)-z$ on $(0,1)$, $f'(1)>1$, and that $q=\lim_n f^{\circ n}(0)$ is the smallest fixed point. Thus the distribution function of $L$ conditional on survival takes only the values zero and one at rational points. The upper bound in Theorem \ref{A3:general} makes it finite; its positive lower bound excludes zero. Consequently the conditional distribution is concentrated at one deterministic positive number.
\end{proof}

\section{Forward profiles and a ray at the limiting bound}\label{forward-proof}
For nearest-neighbor reproduction, the centered tilted walk is simple
symmetric. We keep the endpoint and occupation constraints in the same
block test, so its accepted descendants can be nested.

The nearest-neighbor spine is simple symmetric. Fix $b>1$ and an even integer $K$. Use the corridor $I_K=\{0,\ldots,\lfloor bK\rfloor\}$, with start $0$ and endpoint $K$. Put
\[
 M=\lfloor bK\rfloor+1,\quad D=M+1,\quad
 \lambda=\cos(\pi/D),\quad h(i)=\sin(\pi i/D),\quad i=z+1.
\]
The killed simple-walk matrix $P$ has Doob transform
\[
 Q(i,j)=P(i,j)\frac{h(j)}{\lambda h(i)},\qquad
 \pi_i=\frac{2h(i)^2}{D}.
\]
Write
\[
 A_b=\frac{\pi^2}{2b^2},\qquad
 p_b(z)=\frac2b\sin^2(\pi z/b)\ind_{\{0\le z\le b\}}.
\]
The zero extension is continuous, $K\pi_{z+1}\to p_b(z/K)$ uniformly, and $-K^2\log\lambda\to A_b$.

\begin{lemma}[Forward occupation and endpoint bursts]\label{A4:block}
Fix $b>1$, $\alpha>0$, $\varepsilon>0$, and let $n=2\lfloor\alpha K^3/2\rfloor$. The simple walk has probability at least $cK^{-2}\lambda^n$ of staying in $I_K$, ending at $K$, and satisfying
\begin{equation}\label{A4:occupation}
 \sum_{s=0}^{n-1}\ind_{\{S_s=z\}}
 \ge K^2\{\alpha p_b(z/K)-\varepsilon\}\qquad\hbox{for all integer }z.
\end{equation}
Let $Z_K$ count such branching-walk endpoints. Then
\begin{equation}\label{A4:moments}
 \E Z_K\ge cK^{-2}e^{tK}\lambda^n,\qquad
 \E Z_K^2\le e^{tK}\lambda^n+
 \Psi\sqrt M e^{2tK}\frac{\lambda^{n-1}}{1-\lambda}.
\end{equation}
If $t-A_b\alpha>0$, for every $\delta>0$ and all sufficiently large even $K$,
\begin{equation}\label{A4:bursts}
 \PP\{Z_K\ge\lfloor e^{(t-A_b\alpha-\delta)K}\rfloor\}
 \ge e^{-(A_b\alpha+\delta)K}.
\end{equation}
Thresholds may depend on all fixed parameters.
\end{lemma}

\begin{proof}
We include the uniform occupation argument, since ordinary fixed-state-space ergodicity would not suffice. The exact spectral expansion is
\[
 \frac{Q^n(i,j)}{\pi_j}=\sum_{r=1}^{M}
 \left(\frac{\cos(r\pi/D)}{\cos(\pi/D)}\right)^n
 \frac{\sin(ri\pi/D)\sin(rj\pi/D)}{\sin(i\pi/D)\sin(j\pi/D)}.
\]
The extreme pair contributes two on the accessible parity, and zero on the other parity. For $n\ge AD^2$, the absolute remaining sum is at most
$2\sum_{r\ge2}r^2\exp\{-A(r^2-1)\pi^2/2\}$, which tends to zero as $A\to\infty$. This uses $|\sin(rx)|\le r\sin x$ for integer $r$ and $0<x<\pi$, together with pairing of opposite eigenvalues.

For a fixed block duration $\ell=2\lceil TK^2/2\rceil$, the diagonal Green sum is $O(K)$, uniformly in its site. Indeed its principal pair contributes at most $2(\ell+2)/D$; each other positive mode satisfies
\[
 1-\left(\frac{\cos(r\pi/D)}{\cos(\pi/D)}\right)^2
 \ge\frac{3r^2}{2D^2},\qquad r\ge2.
\]
Pairing modes and summing their even-time geometric series bounds the remainder by $8D/3+2/D$. First hitting a target and then expanding the square of its visit count $R_z$ gives
\[
 \sup_x\E_{Q,x}R_z\le C K,\qquad
 \sup_x\E_{Q,x}R_z^2\le3C^2K^2.
\]
For targets with $p_b(z/K)\ge\eta$, choose $T$ large after a mixing burn-in. Uniformly in the starting state, $\E R_z\ge(1-\eta)\ell\pi_z$. Truncate at $UK$; the mean loss is at most $3C^2K/U$, so a fixed large $U$ makes the truncated conditional mean at least $(1-2\eta)\ell\pi_z$.

Use $J=\lfloor n/\ell\rfloor-1$ blocks and reserve the final duration in $[\ell,2\ell)$. Bounded martingale differences and Azuma's inequality give $\exp(-cJ)$ probability of a deficit below $(1-3\eta)J\ell\pi_z$. Since $J$ is of order $K$, a union over $O(K)$ sites is exponentially small, in particular $o(K^{-1})$. Choose $\eta$ small enough that the uniform limiting profile implies \eqref{A4:occupation}; sites with $p_b<\eta$ need no constraint if $\alpha\eta<\varepsilon$. Rounded boundary sites are absorbed by the same tolerance.

The reserved duration has compatible parity because $K,n,\ell$ are even. Mixing gives endpoint probability at least $\pi_{K+1}\ge c/K$. Intersecting with the occupation event loses only $o(K^{-1})$, so the transformed-walk joint probability remains at least $c'/K$. The original-walk density on these paths is $\lambda^n h(1)/h(K+1)$, and $h(1)/h(K+1)\ge c/K$ for fixed $b>1$. This accounts for the extra polynomial cost at the corridor boundary and proves the walk bound.

Many-to-one gives the first moment in \eqref{A4:moments}. For an ordered endpoint pair branching at generation $j$ and position $x$, conditional expected descendant counts are at most
$e^{t(K-x-Y_i)}\lambda^{n-j-1}$. Indeed, acceptance requires the specified endpoint $K$, and the killed-matrix entry
$P^{n-j-1}(x+Y_i,K)$ is at most its operator norm $\lambda^{n-j-1}$;
a bound for mere corridor survival would not suffice here. Averaging siblings uses exactly $\Psi$. The common-prefix contribution is bounded by
\[
 \E[e^{-tS_j};\tau_{I_K}>j]\le\PP(\tau_{I_K}>j)\le\sqrt M\lambda^j,
\]
because its position is nonnegative. The symmetric matrix $P$ has operator norm $\lambda$, including its negative extreme eigenvalue. Summing over $j$ gives
\[
 \sum_{j=0}^{n-1}\lambda^{2n-j-2}
 \le\frac{\lambda^{n-1}}{1-\lambda},
\]
which proves the second moment. The first and second exponential rates are respectively $t-A_b\alpha$ and $2t-A_b\alpha$. Paley--Zygmund, absorbing polynomial factors into $\delta$, proves \eqref{A4:bursts}.
\end{proof}

\section{A ray at the continuous-profile limit}
We first recall the elementary survival mechanism. Suppose stage $j$
produces at least $R_j$ accepted endpoints with probability at least
$P_j$, where $R_j$ is a finite integer with $R_j\ge1$ and
$0<P_j\le1$ at every stage used. If
\begin{equation}\label{A4:product}
 \sum_{j\ge1}\exp(-P_jR_{j-1})<\infty,
\end{equation}
then the retained branching process survives with positive probability.
The first success has probability at least $P_0>0$.
After a stage-$j-1$ success, the next failure probability is at most
$a_j=\exp(-P_jR_{j-1})<1$. Every initial factor $1-a_j$
is positive. The summability in \eqref{A4:product} and
$\log(1-a_j)\ge-2a_j$ for all sufficiently large $j$ give
$P_0\prod_{j\ge1}(1-a_j)>0$. Endpoint tests expose no reproduction
below their accepted vertices, which are incomparable. Thus future
tests are independent. The finitely branching retained tree has a nested
infinite ray. Thinning by auxiliary independent uniforms realizes the
stated lower probabilities if necessary. A later starting stage is
allowed only after verifying positive-probability access to it.

\subsection{A deterministic phase schedule}
Fix $B>0$ and start phases at an integer $l_0$ such that $Bl_0>1$. For phase $l\ge l_0$ put
\begin{align}
 b_l&=Bl,& \rho_l&=1+1/l,& A_l&=\pi^2/(2B^2l^2),\nonumber\\
 h_l^*&=\frac{tl}{2l+1},&g_l^*&=\frac{t(l+1)}{2l+1},&
 \delta_l&=\frac{t}{100(l+1)^2},\label{A4:phase}\\
 \alpha_l&=\frac{h_l^*-2\delta_l}{A_l},&
 g_l&=g_l^*+\delta_l,&h_l&=h_l^*-\delta_l.\nonumber
\end{align}
These quantities are positive, $g_l$ decreases with $l$, and $h_l$ increases. Apply Lemma \ref{A4:block} with $\varepsilon=1$ and margin $\delta_l$. There is a threshold $T_l$ above which a phase-$l$ block has at least $\lfloor e^{g_lK}\rfloor$ endpoints with probability at least $e^{-h_lK}$. Enlarge thresholds when needed; no uniform assertion in $l$ is being assumed.

In particular $g_l>0$ and $h_l>0$, so every used increment $K>0$
has $\lfloor e^{g_lK}\rfloor\ge1$ and $0<e^{-h_lK}<1$.
Thus the nonzero-start conditions of the survival mechanism hold
throughout all phases, including their finite initial portions.

At stage $j$, with even starting position $H_j$, use the current phase $l_j$ and set
\begin{equation}\label{A4:mesh}
 K_j=2\left\lfloor\frac{H_j}{2l_j}\right\rfloor,
 \qquad H_{j+1}=H_j+K_j.
\end{equation}
Phases are consecutive integers, each lasting a finite number of stages, and each is used for at least ten stages. They can be chosen so that $l_j\to\infty$ and, for all used stages,
\begin{equation}\label{A4:thresholds}
 j\ge10l_j,\qquad K_j\ge\max(T_{l_j},4l_j),\qquad
 \delta_{l_j}K_j\ge j^2+C_0,
\end{equation}
where $C_0$ is any prescribed fixed constant. Here stage indices may start at a sufficiently large integer.

To verify existence, choose an initially large $H_j$ in phase $l_0$. While phase $l$ is held fixed,
\[
 K_{j+1}=\rho_lK_j+e_j,\qquad |e_j|<2,
\]
as follows directly from writing $H_j=lK_j+r$ with $0\le r<2l$. Since $K_j\ge4l$, we have $K_{j+1}\ge(1+1/(2l))K_j$. For $j\ge10l$, $1+1/(2l)>(1+1/j)^2$, so $K_j/(j^2+C_0)$ increases. A sufficiently long phase reaches the next threshold and the next margin requirement. When changing from $l$ to $l+1$, the same division gives $K_{j+1}=K_j$ exactly. Thus the next phase can start with its required bounds. Waiting at least ten stages ensures the first inequality for later phases. Every waiting time is finite, since geometric growth dominates both a fixed threshold and a quadratic function of the stage index.

\begin{remark}
The exact rounding relations use that $H_j$ and $K_j$ are even. At a phase change, $H_{j+1}=(l+1)K_j+r$ with $0\le r<2l<2(l+1)$, so its rounded new increment is precisely $K_j$. Within a phase, the absolute rounding error is less than two. These relations remove a potential unproved uniformity in the changing mesh.
\end{remark}

\subsection{Nested survival at the limiting rates}
Let $R_j=\lfloor e^{g_{l_j}K_j}\rfloor$ and $P_j=e^{-h_{l_j}K_j}$. In a constant phase,
\[
 g_l-\rho_lh_l=(1+\rho_l)\delta_l,
\]
so $g_lK_{j-1}-h_lK_j\ge\delta_lK_j-C_1$, with $C_1$ depending only on $t$. At a phase change, $K_j=K_{j-1}$ and
\[
 g_{l-1}-h_l\ge g_l-h_l=\frac{t}{2l+1}+2\delta_l\ge2\delta_l.
\]
Floors in $R_j$ lose at most a factor two for large increments. Taking $C_0$ larger than the fixed rounding loss in \eqref{A4:thresholds}, we obtain
\[
 P_jR_{j-1}\ge\tfrac12 e^{j^2}
\]
for all sufficiently large stages. This proves \eqref{A4:product}.

An initial chain of positive steps reaches the first $H_j$ with positive probability. Indeed the expected number of $+1$ children is $e^t/2>0$, so such a child has positive probability at each step. Its terminal reproduction is fresh. The surviving retained construction supplies one ray whose stage-$j$ time segment satisfies
\begin{equation}\label{A4:eachprofile}
 N_j^k\ge K_j^2\left\{\alpha_{l_j}
 p_{b_{l_j}}\!\left(\frac{k-H_j}{K_j}\right)-1\right\}.
\end{equation}
Counts exclude each segment's final endpoint, so their time intervals are disjoint.

\subsection{Summing the profiles}
The schedule implies $H_j\to\infty$, $l_j\to\infty$, and
\[
 \frac{K_j}{H_j}\to0,\qquad
 \frac{l_jK_j}{H_j}\to1,\qquad
 \frac{b_{l_j}K_j}{H_j}\to B.
\]
The last two assertions are uniform on every late tail. The dominant part of \eqref{A4:eachprofile} can be written
\begin{equation}\label{A4:coefficient}
 K_j^2\alpha_{l_j}p_{b_{l_j}}\!\left(\frac{k-H_j}{K_j}\right)
 =\left(\frac{2tB}{\pi^2}+o(1)\right)H_jK_j\,
 \sin^2\!\left(\frac{\pi(k-H_j)}{b_{l_j}K_j}\right)
 \ind_{\{H_j\le k\le H_j+b_{l_j}K_j\}}.
\end{equation}
Indeed $4\alpha_l/(2b_l)=2\alpha_l/b_l$ and \eqref{A4:phase} give
$2\alpha_lK_j/(b_lH_j)\to2tB/\pi^2$. This explicitly uses the relative increment $K_j/H_j\sim1/l_j$, not an unsupported interchange of fixed-parameter estimates.

Fix $a\in(0,1/(1+B))$ and sum only over stages with $ak\le H_j\le k$. Earlier and later contributions are nonnegative. The profile is extended by zero, and its scaled sine-square function converges uniformly to the continuous function
\[
 F_B(u)=\sin^2\!\left(\frac{\pi(1-u)}{Bu}\right)
          \ind_{\{1/(1+B)\le u\le1\}},\qquad a\le u\le1.
\]
The truncation error in \eqref{A4:eachprofile}, divided by $k^2$, is at most
\[
 \frac1{k^2}\sum_{ak\le H_j\le k}K_j^2
 \le\left(\max_{ak\le H_j\le k}\frac{K_j}{H_j}\right)
       \frac1{k^2}\sum_{ak\le H_j\le k}H_jK_j=o(1).
\]
The final sum is bounded uniformly, since the stages partition the increasing height axis and $\sum H_jK_j\le\tfrac12(H_{\rm end}^2-H_{\rm start}^2)$. The mesh divided by $k$ tends to zero uniformly in this range. Thus the remaining sum is a Riemann sum and
\begin{align*}
 \liminf_{k\to\infty}\frac{N_\xi^k}{k^2}
 &\ge\frac{2tB}{\pi^2}\int_{1/(1+B)}^1
        u\sin^2\!\left(\frac{\pi(1-u)}{Bu}\right)\,du\\
 &=\frac{2tB}{\pi^2}\int_1^{1+B}
        \frac{\sin^2(\pi(r-1)/B)}{r^3}\,dr.
\end{align*}
This proves the positive-probability construction in \eqref{A4:CB}.

Finite ancestral prefixes and fixed integer spatial translations leave the normalized liminf unchanged. At generation $n$, independent subtree constructions each have the same positive success probability. The original supercritical population tends to infinity on survival, so the probability that none succeeds tends to zero. This upgrades the conclusion to almost surely on $\mathcal S$.

Finally, substituting $r=1+Bz$ gives
\[
 \mathcal C(B)=\frac{2B^2}{\pi^2}\int_0^1
          \frac{\sin^2(\pi z)}{(1+Bz)^3}\,dz.
\]
It is positive and continuous. As $B\downarrow0$, it is at most $B^2/\pi^2$. As $B\to\infty$, splitting at $1/B$ and $1/2$, and using $\sin(\pi z)\le\pi z$ on the first half, bounds it by $O((1+\log B)/B)$. Hence it has a positive maximum at a finite deterministic parameter. Applying the one-ray construction at that parameter proves the maximum assertion in Theorem \ref{A4:main}. Appendix \ref{A4:certificate} proves \eqref{A4:simple}.

\begin{figure}[htbp]
\centering\includegraphics[width=.94\textwidth]{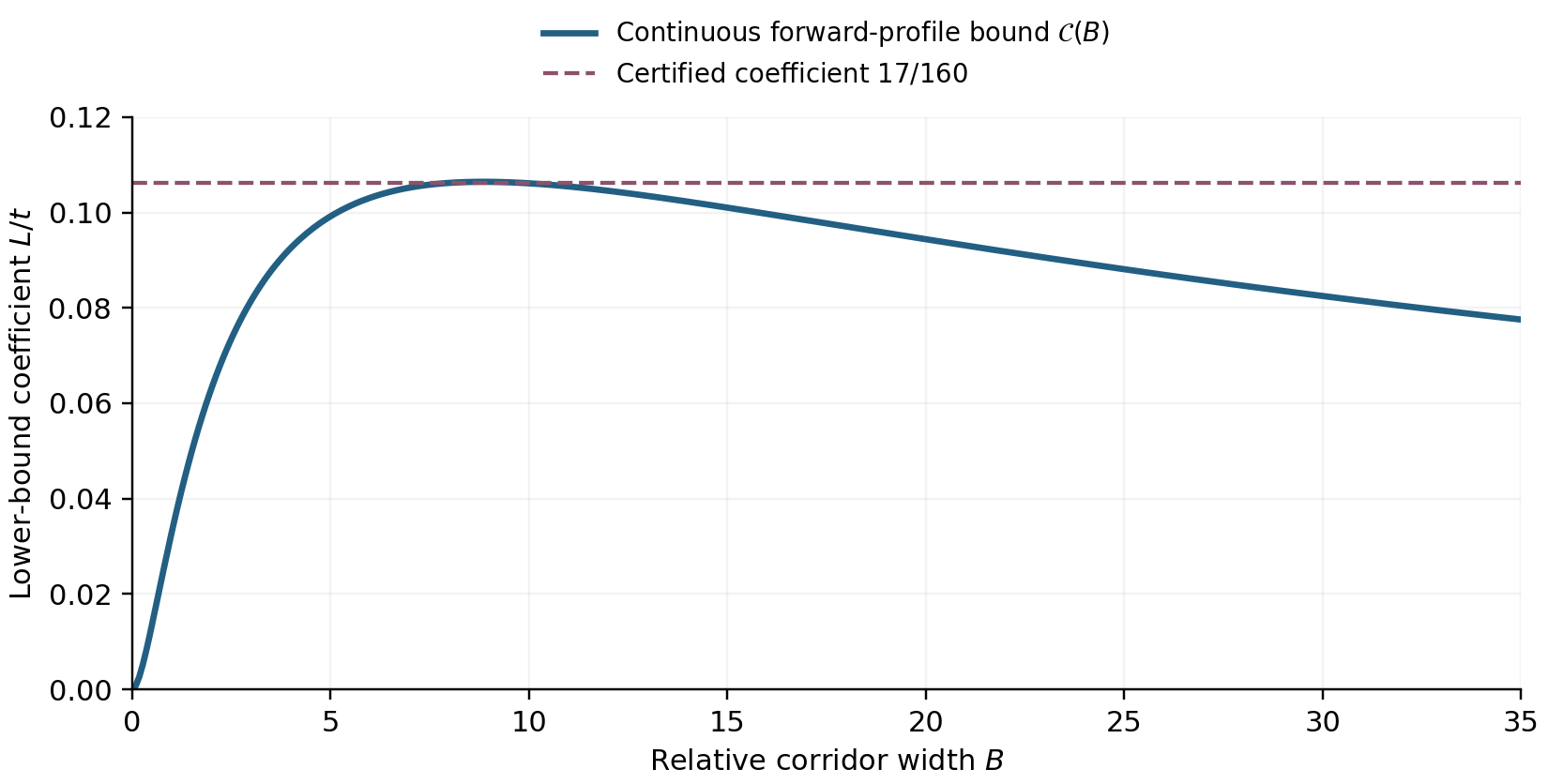}
\caption{Numerical illustration of the sine-profile subfamily. The dashed coefficient has a separate exact certificate. General profiles in Theorem~\ref{A5:main} give a larger bound; this plot does not represent that full variational problem.}
\end{figure}

\section{General forward profiles and their variational optimum}\label{profile-extension}
The sine profile in Section~\ref{forward-proof} is the ground state of an
unweighted killed walk. Allowing a spatial change of measure produces
other profiles, at their Dirichlet energy cost. This cost is familiar
from occupation-measure large deviations \cite{DV}; we give the finite
chain estimates needed here. In particular, the branching second moment
requires a constraint at every prefix, not just a terminal occupation
constraint.

Throughout this section $Y_i\in\{-1,0,1\}$ and the tilted walk has
$Q(1)=Q(-1)=\sigma^2/2$, $Q(0)=1-\sigma^2$, with
$0<\sigma^2\le1$. Thus zero-displacement children are allowed.
Call $f$ admissible on $[0,b]$ if $f\in C^4[0,b]$, $f>0$ on $(0,b)$,
$f(0)=f(b)=f''(0)=f''(b)=0$, and $f'(0)>0>f'(b)$. Set
\[
 F_f=\int_0^b f(x)^2\,dx,\qquad E_f=\int_0^b f'(x)^2\,dx,
 \qquad A_f=\frac{\sigma^2 E_f}{2F_f},\qquad p_f(x)=\frac{f(x)^2}{F_f},
\]
extending $f$ and $p_f$ by zero outside $[0,b]$. Constants in the next proof may
depend on the fixed profile and on its parameters, but not on $K$.

\begin{lemma}[Profile bursts with prefix control]\label{A5:block}
Let $b>1$ and let $f$ be admissible on $[0,b]$. Fix
$\alpha,\varepsilon,\delta>0$ with $t-A_f\alpha>0$.
For all sufficiently large even $K$, there is a test of descendants at
generation $n=2\lfloor\alpha K^3/2\rfloor$, with positions confined to
$\{0,\ldots,\lfloor bK\rfloor\}$ and endpoint $K$, such that every
accepted path has
\[
 \sum_{s=0}^{n-1}\ind_{\{S_s=z\}}
 \ge K^2\{\alpha p_f(z/K)-\varepsilon\}\quad(z\in\mathbb Z).
\]
If $Z_K$ counts accepted endpoints, then
\begin{equation}\label{A5:burst}
 \PP\{Z_K\ge\lfloor e^{(t-A_f\alpha-\delta)K}\rfloor\}
 \ge e^{-(A_f\alpha+\delta)K}.
\end{equation}
The test exposes no reproduction at an accepted endpoint.
\end{lemma}

\begin{proof}
Put $M=\lfloor bK\rfloor+1$, $D=M+1$, $f_i=f(bi/D)$ for
$1\le i\le M$, and $f_0=f_D=0$. Write $X_s=S_s+1$ for the site index,
so that $X_0=1$ and the desired endpoint is $K+1$. Let $P$ be the
tilted-walk kernel killed outside $\{1,\ldots,M\}$, and define
\[
 r_i=\frac{(Pf)_i}{f_i},\qquad
 Q_f(i,j)=\frac{P(i,j)f_j}{(Pf)_i},\qquad
 \pi_i=\frac{f_i(Pf)_i}{\sum_{v=1}^M f_v(Pf)_v}.
\]
The chain $Q_f$ is stochastic and reversible with invariant law $\pi$;
its period is two when $\sigma^2=1$, and one otherwise. Taylor expansion, including at the endpoints, gives
\begin{equation}\label{A5:discrete}
 \sup_i|K^2\log r_i|\le C,
 \qquad K\pi_{z+1}\longrightarrow p_f(z/K)\ \hbox{uniformly},
 \qquad K^2\sum_i\pi_i\log r_i\longrightarrow-A_f.
\end{equation}
Indeed $f''/f$ is bounded, the fourth-order Taylor remainder divided by
$f_i$ is $O(K^{-3})$, and
\[
 \sum_i f_i\{(Pf)_i-f_i\}
 =-\frac{\sigma^2}{2}\sum_{i=0}^{M}(f_{i+1}-f_i)^2.
\]
These facts and Riemann sums prove all three assertions.

Here are uniform estimates for this growing chain. Compare it with the
sine transform $Q_0$ on the same interval. Its eigenvalues are
$\{1-\sigma^2+\sigma^2\cos(r\pi/D)\}/\lambda_K$, where
$\lambda_K=1-\sigma^2+\sigma^2\cos(\pi/D)$; for $\sigma^2=1$
this is the chain in Lemma~\ref{A4:block}.
The ratios $f_i/\sin(\pi i/D)$ are bounded above and away from zero.
Consequently the stationary measures and the edge conductances
$\pi_iQ_f(i,i+1)$ are comparable to those of $Q_0$. Comparing their
Dirichlet forms gives a spectral gap at least $cK^{-2}$ below the
eigenvalue $1$. When $\sigma^2=1$, bipartiteness pairs each
eigenvalue with its negative. When $\sigma^2<1$, the diagonal entries
of $Q_f$ are bounded below by a fixed positive constant, so its
eigenvalues are bounded away from $-1$. Thus all eigenvalues other
than $1$ (and $-1$ in the period-two case) have absolute value at most
$1-cK^{-2}$. Also $\min_i\pi_i\ge cK^{-3}$. The reversible spectral
expansion therefore gives, on accessible parity, with $d=2$ if
$\sigma^2=1$ and $d=1$ otherwise,
\begin{equation}\label{A5:mixing}
 \left|\frac{Q_f^m(i,j)}{\pi_j}-d\right|
 \le CK^3e^{-cm/K^2}.
\end{equation}

We also need a sharper return estimate than the gap alone supplies.
For $\zeta=e^{-K^{-2}}$, the variational formula
\[
 \sum_{m\ge0}\zeta^m Q_f^m(i,i)
 =\pi_i\sup_{g\ne0}
 \frac{g(i)^2}{(1-\zeta)\|g\|_{\pi}^2+
                  \zeta\langle g,(I-Q_f)g\rangle_{\pi}}
\]
and the same form comparison bound this resolvent by a constant times
the sine-chain resolvent. The latter is
\[
 \frac2D\sum_{r=1}^{M}
 \frac{\sin^2(r\pi i/D)}
 {1-\zeta\{1-\sigma^2+\sigma^2\cos(r\pi/D)\}/\lambda_K}\le CK.
\]
For $r\le D/2$ the denominator is at least $cr^2/K^2$;
the remaining terms have denominator bounded below by a positive
constant depending on $\sigma^2$.
First hitting $i$ now gives
$\sup_x\E_{Q_f,x}\sum_{m<K^2}\ind_{\{X_m=i\}}\le CK$.
Splitting a longer interval into pieces of this length, and expanding
the square by its earlier visit, shows that for
$\ell=2\lceil T K^2\log K/2\rceil$ and
$R_i=\sum_{m<\ell}\ind_{\{X_m=i\}}$,
\begin{equation}\label{A5:visits}
 \sup_x\E_x R_i\le C_TK\log K,\qquad
 \sup_x\E_x R_i^2\le C_T(K\log K)^2.
\end{equation}

Fix a small tolerance $\eta>0$. After a burn-in of $C K^2\log K$
steps, \eqref{A5:mixing} has arbitrarily small relative error. Taking
$T$ large relative to this burn-in makes
$\E_x R_i\ge(1-\eta)\ell\pi_i$ uniformly in $x$.
In the period-two case the parities have equal invariant mass and
an even block spends equal numbers of steps on them; in the aperiodic
case the limiting density in \eqref{A5:mixing} is simply one. For sites with $p_f((i-1)/K)\ge\eta$,
truncating $R_i$ at $UK\log K$ loses at most
$C_TK\log K/U$ in conditional mean by \eqref{A5:visits}.
Choose $U$ large. Apply bounded martingale differences to these
truncated counts over $J=\lfloor n/\ell\rfloor-1$ blocks.
Since $J\asymp K/\log K$, the chance of an occupation deficit of any
fixed positive multiple of $K^2$ at any site is at most
$CK\exp\{-cK/\log K\}$. Sites with sufficiently small $p_f$ need
no lower constraint. Choosing the tolerances in terms of $\varepsilon$
proves the desired occupation inequalities except on this event.

The same block argument controls the likelihood on every prefix.
The function $v_K(i)=K^2\log r_i$ is uniformly bounded. By
\eqref{A5:mixing}, the conditional expected sum of $v_K$ in an even
block differs from $\ell\pi(v_K)$ by at most an arbitrarily small
multiple of $\ell$, uniformly in the starting state: discard the
burn-in and use the stationary average on the rest (the two parity
averages together in the period-two case). The centered
block sums have absolute value at most $C\ell$.
Azuma's inequality followed by a union over the $J$ block endpoints,
\eqref{A5:discrete}, and the $O(\ell)$ bound inside a block imply
that, for any fixed $\rho>0$,
\begin{equation}\label{A5:prefix}
 \max_{0\le m\le n}
 \left|\sum_{s<m}\log r_{X_s}+\frac{A_fm}{K^2}\right|
 \le\rho K
\end{equation}
except with probability $CK\exp\{-cK/\log K\}$. Here choose the
block-mean tolerance sufficiently small in terms of $\rho,\alpha$;
the within-block contribution after division by $K^2$ is
$O(\log K)=o(K)$. The unused final interval, of length in
$[\ell,2\ell)$, is absorbed by the same estimate. Its mixing bound
gives $\PP_{Q_f}(X_n=K+1)\ge c/K$. Subtracting the preceding failure
probabilities leaves at least $c'/K$ for endpoint, occupation and
prefix constraints jointly.

For every path segment the exact change of measure is
\begin{equation}\label{A5:density}
 \frac{\PP_P(X_0,\ldots,X_m)}{\PP_{Q_f}(X_0,\ldots,X_m)}
 =\frac{f_{X_0}}{f_{X_m}}\prod_{s<m}r_{X_s}.
\end{equation}
Accept precisely the paths with the three constraints just obtained.
The boundary-to-endpoint factor is at least $c/K$, so many-to-one gives
\[
 \E Z_K\ge cK^{-2}
       \exp\{tK-A_fn/K^2-\rho K\}.
\]
For the second moment, subtracting two prefix inequalities in
\eqref{A5:prefix} bounds the likelihood of every accepted suffix by
its length. The ratios in \eqref{A5:density} cost at most $CK$,
since $\max_i f_i/\min_i f_i\le CK$.

To display the branching estimate, let $B_j$ be the event that a
length-$j$ prefix stays in the corridor and obeys the likelihood bound
in \eqref{A5:prefix} at every time through $j$. The preceding density
bound implies, for the tilted walk $P$,
\[
 \E_P[e^{-tS_j};B_j]\le\PP_P(B_j)
 \le CK\exp\{-A_fj/K^2+\rho K\},
\]
because $S_j\ge0$ on $B_j$. Given a splitting vertex with such a
prefix and its offspring list, the conditional expected number of
accepted continuations through a specified child $i$ is at most
\[
 CK\exp\!\left\{t(K-S_j-Y_i)
       -\frac{A_f(n-j-1)}{K^2}+2\rho K\right\}.
\]
The bound is uniform in the child's corridor position; for a child
outside the corridor the count is zero. Conditional independence of
the two descendant trees and averaging ordered distinct children give
\[
 \E[Z_K(Z_K-1)]\le
 C\Psi K^2e^{2tK+4\rho K}
 \sum_{j=0}^{n-1}e^{-2A_f(n-j-1)/K^2}
              \E_P[e^{-tS_j};B_j].
\]
The common-prefix term is a weighted expectation under the original
tilted walk $P$, bounded above by its prefix probability as shown
above. The chain $Q_f$ supplies the change of measure. Since
$2n-j-2\ge n-1$, the sum is at most
$CKn\exp\{-A_f(n-1)/K^2+\rho K\}$. Including the diagonal gives,
for some fixed $C,d$,
\[
 \E Z_K^2\le CK^d\left[
 e^{tK-A_fn/K^2+\rho K}
 +e^{2tK-A_fn/K^2+5\rho K}\right].
\]
Choose $\rho>0$ with $7\rho<\min\{\delta,t-A_f\alpha\}$.
The first-moment lower bound implies, for large $K$,
$\frac12\E Z_K\ge e^{(t-A_f\alpha-\delta)K}$.
The two moment estimates and Paley--Zygmund give
\[
 \PP\{Z_K\ge\tfrac12\E Z_K\}
 \ge cK^{-d-4}\exp\{-A_fn/K^2-7\rho K\}
 \ge e^{-(A_f\alpha+\delta)K},
\]
after increasing $K$ again. This proves \eqref{A5:burst}, including
its integer threshold. All tests concern paths up to generation $n$;
the endpoints' own reproduction is unused.
\end{proof}

\begin{lemma}[The forward-profile variational problem]\label{A5:variational}
Let $j_{1,1}$ be the first positive zero of the Bessel function $J_1$.
Then
\begin{equation}\label{A5:quotient}
 \sup_{B>0}\ \sup_{0\ne f\in H_0^1(0,B)}
 \frac{\displaystyle\int_0^B f(x)^2(1+x)^{-3}\,dx}
      {\displaystyle\int_0^B f'(x)^2\,dx}
 =\frac4{j_{1,1}^2}.
\end{equation}
The supremum can be approached by admissible profiles whose energies
and whose restrictions to each compact interval converge to those of
\[
 \phi(x)=\sqrt{1+x}\,J_1\!\left(\frac{j_{1,1}}{\sqrt{1+x}}\right).
\]
\end{lemma}
\begin{proof}
Put $\lambda_*=j_{1,1}^2/4$. Bessel's equation and its expansion at
zero \cite[Sections 10.2 and 10.21]{DLMF} give
\begin{gather*}
 \phi(0)=0,\qquad \phi(x)>0\ (x>0),\qquad
 -\phi''(x)=\frac{\lambda_*\phi(x)}{(1+x)^3},\\
 \phi(x)\longrightarrow j_{1,1}/2,\qquad \phi'(x)=O((1+x)^{-2}).
\end{gather*}
The zero at the origin is simple and $\phi''(0)=0$. For a smooth
compactly supported function $f$ vanishing at zero, integration by
parts gives the ground-state identity
\begin{equation}\label{A5:groundstate}
 \int_0^\infty f'^2-\lambda_*\int_0^\infty\frac{f^2}{(1+x)^3}
 =\int_0^\infty\phi^2\left((f/\phi)'\right)^2\ge0.
\end{equation}
The boundary term at zero vanishes because $f(x)=O(x)$ and
$\phi(x)$ has a simple zero. Density proves the upper bound in
\eqref{A5:quotient} for $H_0^1(0,B)$.

For completeness the approximations can satisfy all the regularity
conditions of Lemma~\ref{A5:block}. Choose a smooth function $\chi$
on $[1,2]$, equal to one near $1$, positive before $2$, and equal to
$2-s$ near $2$. Let $f_R=\phi$ on $[0,R]$. On $[R,2R]$ use
$\phi(x)\chi(x/R)$ and, on a fixed rescaled neighborhood of $2R$,
subtract
\[
 \frac{\phi'(2R)\chi'(2)}{R}(x-2R)^2\psi(x/R),
\]
where $\psi$ is smooth, equals one near $2$, and vanishes outside
that neighborhood. This cancels $f_R''(2R)$, changes the function
by $O(R^{-1})$, and changes its derivative by $O(R^{-2})$.
For large $R$ it preserves positivity: the main function vanishes
linearly at $2R$, whereas the correction vanishes quadratically,
with relative size $O(R^{-1})$. Thus $f_R$ is admissible on $[0,2R]$,
equals $\phi$ on $[0,R]$, and
\[
 E_{f_R}\longrightarrow E_\phi:=\int_0^\infty\phi'^2\in(0,\infty),
 \qquad \int_0^{2R}\frac{f_R^2}{(1+x)^3}\,dx
 \longrightarrow\int_0^\infty\frac{\phi^2}{(1+x)^3}\,dx.
\]
Indeed the added tail energy is $O(R^{-1})$ and the profiles are
uniformly bounded. Integration of the differential equation against
$\phi$ gives $E_\phi=\lambda_*\int\phi^2(1+x)^{-3}$; both boundary
terms vanish. This proves the lower bound and the approximation assertion.
The variational identity is a classical one-dimensional ground-state
calculation; the application to one persistent branching-walk ray is
the point of its use here.
\end{proof}

\begin{proof}[Proof of Theorem~\ref{A5:main}]
First fix an admissible profile $f$ on $[0,B]$ and start at a phase
$l_0$ with $Bl_0>1$. In phase $l$ use
$f(x/l)$ on $[0,Bl]$, whose energy cost and density are
\[
 A_l=\frac{\sigma^2 E_f}{2l^2F_f},\qquad
 p_l(x)=\frac{f(x/l)^2}{lF_f}.
\]
Keep $h_l^*,g_l^*,\delta_l,g_l,h_l$ from \eqref{A4:phase} and put
$\alpha_l=(h_l^*-2\delta_l)/A_l$.
Lemma~\ref{A5:block}, with occupation tolerance one and burst margin
$\delta_l$, supplies exactly the same burst exponents $g_l,h_l$.
The deterministic phase schedule \eqref{A4:mesh}--\eqref{A4:thresholds}
and its survival proof apply unchanged, after enlarging each fixed-phase
threshold. On the resulting ray put $q_j=l_jK_j/H_j\to1$. The exact
cancellation of $F_f$ gives the main occupation contribution as
\begin{align*}
 K_j^2\alpha_{l_j}p_{l_j}\!\left(\frac{k-H_j}{K_j}\right)
 &=\frac{2(h_{l_j}^*-2\delta_{l_j})q_j}{\sigma^2 E_f}
 H_jK_j f\!\left(\frac{k-H_j}{l_jK_j}\right)^2\\
 &=\left(\frac{t}{\sigma^2 E_f}+o(1)\right)H_jK_j
 f\!\left(\frac{k-H_j}{l_jK_j}\right)^2.
\end{align*}
The same Riemann sum and tolerance estimate as in
\eqref{A4:coefficient} give
$\liminf_k N_\xi^k/k^2\ge t\int f^2(1+x)^{-3}/(\sigma^2 E_f)$.
In particular, the old sine family is the restriction
$f(x)=\sin(\pi x/B)$ of this quotient.

To obtain one ray at the supremum, take the profiles $f_l=f_{R_l}$
of Lemma~\ref{A5:variational}, where $R_l\to\infty$ deterministically.
In phase $l$ use $f_l(x/l)$, now on $[0,2lR_l]$, and set
$A_l=\sigma^2 E_{f_l}/(2l^2F_{f_l})$. All parameters are fixed within a phase;
no estimate uniform in the profile or in its support is required.
The same formulas give the same $g_l,h_l$, so the previous scheduling
and survival proof still work. The displayed exact cancellation also
holds with $E_f$ replaced by $E_{f_{l_j}}$; hence growth of $F_{f_l}$
introduces no additional limiting factor. For any fixed $a\in(0,1)$, sum only
over $ak\le H_j\le k$. On this range, the profile arguments lie in
a fixed compact interval for large $k$. There $f_{l_j}=\phi$
eventually and $E_{f_{l_j}}\to E_\phi$. Consequently the Riemann sum
and the vanishing error $k^{-2}\sum K_j^2=o(1)$ give
\[
 \liminf_{k\to\infty}\frac{N_\xi^k}{k^2}
 \ge\frac{t}{\sigma^2 E_\phi}\int_a^1 u\phi(u^{-1}-1)^2\,du.
\]
This is the same constructed ray for every $a$. Let $a\downarrow0$
to obtain $t\int_0^\infty\phi(x)^2(1+x)^{-3}\,dx/(\sigma^2 E_\phi)
=4t/(\sigma^2 j_{1,1}^2)$. Initial access by positive steps and amplification
on survival are as in Section~\ref{forward-proof}; here the expected
number of $+1$ children is $e^t\sigma^2/2>0$.

An elementary comparison is obtained from $f(x)=x/(1+x)$:
$\int_0^\infty f'^2=1/3$ and
$\int_0^\infty f^2(1+x)^{-3}=1/12$. Compact truncation gives
the coefficient $1/4$. It is strictly below $4/j_{1,1}^2$:
the ground-state identity is strict for this $f$, which is not
proportional to $\phi$, and its boundary term at infinity also vanishes.
No numerical approximation is needed for this comparison.
\end{proof}

\section{Ladder potentials for real displacements}\label{renewal-proof}

Kill the walk on entering $(-\infty,0)$ and define its occupation measure
\[
 G(x,B)=\E_x\sum_{n<\tau_-}\ind_{\{S_n\in B\}},\qquad x\ge0.
\]
Let $U_-$ be the renewal measure of the strict descending ladder heights,
and $U_+$ that of the weak ascending ladder heights. These measures include
the initial atom at zero. A missing ladder epoch is represented by a
defective height law, or equivalently a jump to a cemetery state $+\infty$.
Weak ascending epochs use $S_n\ge S_{\rm current}$; their height law may
have an atom at zero. Strict descending heights are positive.

\begin{lemma}[Factorization and interval bounds]\label{C6:potential}
For every nonzero increment law and every $b>0$,
\begin{align}
 G(0,dy)&=U_+(dy),\label{C6:rootpotential}\\
 G(x,B)&=\int_{[0,x]}U_-(dh)\,
 U_+\bigl(\{z\ge0:x-h+z\in B\}\bigr),\label{C6:factor}\\
 G(x,[y,y+b))&\le C_b(1+x),\qquad x,y\ge0.\label{C6:bandupper}
\end{align}
In particular $G(0,[y,y+b))\le C_b$, uniformly in $y$.
These assertions need no drift or moment assumption.
\end{lemma}
\begin{proof}
A renewal measure $U$ with nonnegative, possibly defective increments and
with mass at zero strictly less than one has finite mass on each compact
interval. To see this, choose $\delta>0$ such that an increment is at least
$\delta$ or is killed with a fixed positive probability. The time to collect
more than $b/\delta$ such successes is integrable and dominates the number
of renewals before leaving $[0,b]$. Moreover, stop at the first renewal at
or above $y$. If it lands in $[y,y+b]$, the remaining expected number in
this interval is at most $U([0,b])$. Thus
\begin{equation}\label{C6:renewupper}
 \sup_{y\ge0}U([y,y+b])\le U([0,b])<\infty.
\end{equation}
The same bound applies to intervals meeting the negative half-line after
intersection with $[0,\infty)$.

The strict ladder law has no atom at zero. For the weak ascending law,
mass one at zero would force every positive step to have probability zero,
and then every negative step to have probability zero as well: a negative
first step could never return to zero without a positive step. This is
excluded by $Q\ne\delta_0$. Consequently \eqref{C6:renewupper} applies to
both ladder measures. It gives $U_-([0,x])\le C(1+x)$.

For a path of length $n$ starting at zero, reverse the order of its
increments. The original partial sums are nonnegative exactly when the
reversed endpoint is a weak record maximum. Summing over $n$ gives
\eqref{C6:rootpotential}: all weak record times are exactly the successive
weak ascending ladder epochs, including time zero.

For a path starting at $x$, split it at the first time of its global
minimum over its finite duration. The prefix endpoint is a strict
descending record with drop $h\le x$, and the suffix, translated to zero,
stays nonnegative. Conversely these two pieces concatenate uniquely to
such a path. Independence of disjoint increment blocks and summation over
their lengths give \eqref{C6:factor}. The choice of the first minimum is
essential when ties occur. Apply \eqref{C6:renewupper} to the inner measure
and the bound on $U_-([0,x])$ to prove \eqref{C6:bandupper}.
\end{proof}

\begin{lemma}[Root lower bound]\label{C6:lower}
Assume \eqref{C6:ladder}. If the closed additive group generated by the
increment support is $\mathbb R$, fix $b=1$. If it is $h\mathbb Z$,
fix $b=h$. There exist $c,C>0$ and $k_0$ such that, for
$I_k=[kb,(k+1)b)$,
\begin{equation}\label{C6:rootbounds}
 c\le G(0,I_k)\le C\quad(k\ge k_0),\qquad
 G(x,I_k)\le C(1+x)\quad(x\ge0).
\end{equation}
Condition \eqref{C6:onesided} implies \eqref{C6:ladder}.
\end{lemma}
\begin{proof}
Let $p_0$ be the probability that the first weak ascending height is zero.
Under \eqref{C6:ladder} the weak height law is proper, $p_0<1$, and its
positive part, conditional on being positive, is the law of $H^+$.
Deleting successive zero heights therefore gives
\begin{equation}\label{C6:weakstrict}
 U_+=\frac{1}{1-p_0}\sum_{n\ge0}\mathcal L(H^+)^{*n}.
\end{equation}
Its mean renewal increment is $(1-p_0)\E H^+\in(0,\infty)$.

The spatial arithmetic groups of $H^+$ and $S_1$ coincide. Indeed, if
$H^+$ is supported on $h\mathbb Z$, every positive step is in
$h\mathbb Z$ because it is an immediate strict ascent. There is then a
positive atom $a\in h\mathbb Z$. Follow any negative step $-z$ by enough
copies of $a$ to make the first strict ascent. Its height is $ja-z$ in
$(0,a]$. If a positive-probability set of $z$ were outside $h\mathbb Z$,
partitioning it according to the required integer $j$ would contradict the
height support. Thus negative steps also lie in $h\mathbb Z$.
Conversely, steps on $h\mathbb Z$ give heights on $h\mathbb Z$.

Blackwell's renewal theorem \cite[Theorem 2.6.4]{Durrett} gives a positive
limit for the mass of every fixed-length interval in the nonarithmetic
case. Half-open intervals have the same limit: bound them by closed
intervals of lengths $b-\varepsilon$ and $b$, then let
$\varepsilon\downarrow0$. In the arithmetic case the renewal mass at
$kh$ tends to $h/\E H^+$. Together with \eqref{C6:weakstrict} these limits
give the lower bound in \eqref{C6:rootbounds}. The upper bounds follow from
Lemma~\ref{C6:potential}. No density or spread-out assumption has been used.

Finally, an integrable nonzero walk of nonnegative mean drifts to $+\infty$
or oscillates, so $T^+$ is proper. Before this first strict ascent it stays
in $(-\infty,0]$. Let $\widehat G$ be the killed occupation measure of
the reflected walk $-S$, killed below zero. The overshoot decomposition is
\[
 \E H^+=\int_{[0,\infty)}\widehat G(0,dz)\,
                 \E[(S_1-z)^+].
\]
The root interval bound of Lemma~\ref{C6:potential}, applied to $-S$, gives
\[
 \E H^+\le C\sum_{m\ge0}\E[(S_1-m)^+]
 \le C'\E\bigl[S_1^++(S_1^+)^2\bigr]<\infty.
\]
This uses only the positive second moment. Positivity follows from strict
ascent.
\end{proof}

\section{Uniform convergence from a finite critical moment}
We give the band argument in detail because discretizing a real branching
random walk by rounding its increments would change its normalization.
Here the original walk is retained throughout. In this section any fixed
$b>0$ may be used; write $I_k=[kb,(k+1)b)$. Let $a$ range over nonnegative integer multiples of $b$,
and let $\T_a$ keep vertices whose ancestral positions are all at least
$-a$. The shifted potential is $G_a(x,B)=G(x+a,B+a)$.
Many-to-one gives
\begin{equation}\label{C6:manyone}
 \E_x\sum_{u\in\T_a}f(V(u))
   =\int e^{t(y-x)}f(y)G_a(x,dy).
\end{equation}
The expected number of vertices on the first-crossing line below $-a$ is
at most $e^{-ta}$. Consequently
\begin{equation}\label{C6:localize}
 \PP(\T\ne\T_a)\le e^{-ta},\qquad
 \PP(\T=\T_a\text{ for some }a)=1.
\end{equation}
This stopping-line formula follows by summing the ordinary many-to-one
formula over the first-crossing generation. It does not assume the
tilted walk crosses the barrier almost surely.

\subsection{Uniform local time in bands}
Every band of $\T_a$ contains finitely many vertices almost surely, by
\eqref{C6:manyone} and Lemma~\ref{C6:potential}. Starting from any point of
$I_k$, the expected number of visits to $I_k$ before killing is at most
$D_k=C_a(k+a/b+2)$, for $k\ge-a/b$. At any visit time the same bound
holds for future visits. Markov's inequality, followed by the strong
Markov property at successive visit times, therefore yields
\[
 \sup_{x\in I_k}\PP_x\{L(I_k)\ge r\}
 \le C\exp(-c r/D_k).
\]
For example, use blocks of $\lceil2D_k\rceil$ further visits, whose
conditional probability is at most $1/2$. The stopping line of $r$th
band visits in the tree has expected size at most
$e^{t(k+1)b}C\exp(-c r/D_k)$, by many-to-one. Choose
$r=C'_a(k+a/b+2)^2$ with a sufficiently large constant. Summation over
$k$ and Borel--Cantelli give a finite random $C_a'$ with
\begin{equation}\label{C6:bandlocal}
 \sup_\pi\#\{u\in\pi\cap\T_a:V(u)\in I_k\}
 \le C_a'(k+a/b+2)^2.
\end{equation}
The finitely many exceptional bands are absorbed using their finite
populations. The supremum is over finite and infinite paths.

\subsection{Comparable large marks are rare}
Put $m_t=\E\eta^t<\infty$. For $d>0$ let
$p_d(k)=\PP(\eta>d e^{kb})$. Geometric summation gives
\begin{equation}\label{C6:tailgrid}
 H(d):=\sum_{k\ge0}e^{tkb}p_d(k)\le C_{t,b}d^{-t}m_t,
 \qquad p_d(k)\le d^{-t}m_t e^{-tkb}.
\end{equation}
Let $P_a(d,K)$ count strict ancestor--descendant pairs with both vertices
at positions at least $Kb$ and both amplitudes greater than $d$.
Applying \eqref{C6:manyone} first to the ancestor and then to the descendant,
and using the upper endpoints only to bound exponential counting weights,
gives for independent marks
\begin{align}
 \E P_a(d,K)
 &\le C_a H(d)\sum_{j\ge K}(j+a/b+2)p_d(j)\notag\\
 &\le C_a' m_t^2d^{-2t}(K+a/b+2)e^{-tbK}.\label{C6:pairind}
\end{align}
The factors $e^{\pm tb}$ from band endpoints are fixed constants. For
clarity, at an ancestor of position $x$, the expected marked descendant
count in $I_k$ is bounded by
$C(1+x+a)e^{t((k+1)b-x)}p_d(k)$; multiplication by the ancestor's
counting weight $e^{tx}$ leaves exactly the summands above.

For dependent marks reveal the ancestor's whole pair. A child at $x+Y_i$
which is not killed contributes at most
$C(1+x+a)(1+Y_i^+)e^{-t(x+Y_i)}H(d)$ in expectation. Hence
\begin{equation}\label{C6:pairdep}
 \E P_a(d,K)\le C_a H(d)\sum_{j\ge K}(j+a/b+2)
            \E[Z\ind_{\{\eta>d e^{jb}\}}].
\end{equation}
With $\alpha=1-1/p>0$, H\"older and \eqref{C6:tailgrid} bound this by
\begin{equation}\label{C6:pairholder}
 C_a\|Z\|_p m_t^{1+\alpha}d^{-t(1+\alpha)}
       (K+a/b+2)e^{-\alpha tbK}.
\end{equation}
The child's mark is fresh relative to its ancestral history, even when it
depends on the child's own reproduction.

Choose $d_n=2^{-n}$ and $K_n=\lceil cn\rceil$, with
$bc>2\log2$ in \eqref{C6:pairind}, or
$bc>(1+\alpha)\log2/\alpha$ in \eqref{C6:pairholder}. The probabilities of
a comparable pair are summable. Almost surely for all large $n$, a path
contains at most one vertex above $K_nb$ with amplitude greater than
$2^{-n}$. Partition small positive amplitudes into
$(2^{-(n+1)},2^{-n}]$. Below $K_{n+1}b$, \eqref{C6:bandlocal} bounds the
number of vertices in the class by a random constant times
$(n+a/b+2)^3$; above it there is at most one. Thus
\begin{equation}\label{C6:uniformsmall}
 \sup_\pi\sum_{\substack{u\in\pi\cap\T_a\\0<A(u)\le2^{-M}}}A(u)
 \le C_a\sum_{n\ge M}2^{-n}(n+a/b+2)^3\longrightarrow0.
\end{equation}
There are only finitely many vertices of $\T_a$ with $A(u)>d$ for each
$d>0$: use \eqref{C6:tailgrid} for nonnegative bands and the finite
populations in the finitely many negative bands. Remove a generation
containing all those vertices and use \eqref{C6:uniformsmall}. This proves
uniform genealogical tail convergence. Localization \eqref{C6:localize}
proves Theorem~\ref{C6:suff}.

To obtain attainment, append an absorbing symbol of zero contribution to
each finite leaf path. The resulting boundary is compact because every
vertex has finitely many children. The truncated sums are continuous
cylinder functions, and their uniform limit is continuous. Infinite rays
form a closed subset, nonempty on survival. Maxima exist on both spaces.

The finite supremum also gives the minimal nonnegative endogenous solution
 of the recursion $X\stackrel d=\eta+\max_i e^{-Y_i}X_i$, with empty
 maximum zero. Iteration from zero produces the finite-depth sums and their
 increasing limit. Any other nonnegative solution stochastically dominates
 every such iterate, and therefore dominates the limit. This does not
 establish uniqueness.

\section{The converse under the ladder condition}
Choose the bands of Lemma~\ref{C6:lower}. For any sequence
$g_k\in[0,1]$ with $\sum_k g_k=\infty$, put
\[
 A_R=\sum_{k=0}^R g_k,\qquad
 f_R(v)=\sum_{k=0}^R g_k\ind_{\{v\in I_k\}},\qquad
 T_R=\sum_{u\in\T_0}e^{-tV(u)}f_R(V(u)).
\]
For sufficiently large $A_R$, the root band bounds give
\begin{equation}\label{C6:firstmoment}
 c A_R\le\E T_R\le C A_R.
\end{equation}
The omitted first $k_0$ bands have total at most $k_0$, so constants may be
chosen independently of the sequence $g$. For $x\ge0$, define
\[
 F_R(x)=\int f_R(y)G(x,dy)\le C(1+x)A_R.
\]
The expected absolute weighted contribution of a tree starting at $x$
is $e^{-tx}F_R(x)$.

Expand $T_R^2$ into identical vertices, strict ancestor pairs in both
orders, and pairs descending from distinct children of their last common
ancestor. The first two terms are bounded by
\[
 C\sum_{k\ge0}e^{-tkb},\qquad
 C A_R\sum_{k\ge0}e^{-tkb}(k+1),
\]
respectively. For the third, a surviving child of an ancestor at $x$ has
expected contribution at most
$C A_Re^{-tx}(1+x)w_i$.
Different child trees are independent conditionally on the full offspring
list. Many-to-one at the common ancestor therefore bounds the third term by
\[
 C\Omega A_R^2\sum_{k\ge0}e^{-tkb}(k+1)^2.
\]
Finite-depth truncation and Tonelli justify the expansion. No independence
between siblings and no diagonal second moment of $Z$ were used. Thus
\begin{equation}\label{C6:secondmoment}
 \E T_R^2\le C' A_R^2\quad(A_R\ge1).
\end{equation}
Paley--Zygmund and monotonicity imply
\begin{equation}\label{C6:infiniteoccupation}
 \PP(T_\infty=\infty)\ge c_*>0,
\end{equation}
with $c_*$ independent of $g$.

If $\E\eta^t=\infty$, then for every $M>0$
\[
 \sum_{k\ge0}e^{tkb}\PP(\eta>M e^{(k+1)b})=\infty.
\]
This follows by comparing the tail integral for $\E\eta^t$ with its
geometric grid sums; shifting one index only changes a fixed factor.
Take
$g_k=\min\{1,e^{tkb}\PP(\eta>M e^{(k+1)b})\}$.
For $v\in I_k$,
\[
 e^{-tv}g_k\le\PP(\eta>M e^{(k+1)b})
                 \le\PP(\eta>M e^v).
\]
The upper endpoint in this threshold is important. Equation
\eqref{C6:infiniteoccupation} gives, with probability at least $c_*$,
\begin{equation}\label{C6:compensator}
 \sum_{u\in\T_0}\PP(\eta>M e^{V(u)})=\infty.
\end{equation}
Explore the killed tree breadth first, revealing each mark and its own
reproduction together. Before that revelation, existence and position are
known and the local pair is fresh. Thus the terms in \eqref{C6:compensator}
are the predictable exceedance probabilities. Conditional Borel--Cantelli
gives infinitely many exceedances whenever their sum diverges. One direct
proof uses the nonnegative mean-one martingale
\[
 \frac{\exp(-\sum_{j\le n}I_j)}
 {\prod_{j\le n}[1-(1-e^{-1})p_j]},
\]
where $I_j$ are the successive exceedance indicators and $p_j$ their
conditional probabilities. On $\{\sum p_j=\infty,\sum I_j<\infty\}$ it
would diverge to infinity, contradicting almost-sure finite convergence.
Append zero-probability dummy trials if the queue empties.

Consequently $\PP(\sup_u A(u)>M)\ge c_*$ uniformly in $M$.
Letting $M\uparrow\infty$ gives positive probability of divergence.
Finiteness of this supremum is equivalent to finiteness in every child
subtree: the root mark, offspring number, and spatial multipliers are all
finite. Its probability $r$ therefore solves $r=\E r^N$. For a
supercritical Galton--Watson law the only fixed points in $[0,1]$ are $q$
and $1$, also when $\E N=\infty$. Positive divergence excludes $r=1$;
extinction gives a finite tree. This proves Theorem~\ref{C6:nec}.

\section{Surviving rays and a strict infinite-variance example}
\subsection{The backbone}
Under the law conditioned on survival, keep only children with surviving
descendant trees. Given the offspring list, their statuses are independent
Bernoulli$(s)$ variables $B_i$, conditioned on at least one success. The
backbone local law is
\[
 \E_\infty F=\frac1s\E\left[
 F(\eta,(Y_i:B_i=1))\ind_{\{\sum B_i\ge1\}}\right].
\]
Its child subtrees are independent copies of the tree conditioned to
survive. Recursion gives iid local pairs on a no-leaf branching tree.
Direct calculation gives
\begin{align*}
 \E_\infty\sum_{i:B_i=1}f(Y_i)&=\E\sum_i f(Y_i),\\
 \Omega_\infty=s\Omega,\qquad
 \E_\infty Z_\infty^p&\le s^{-1}\E Z^p,\\
 \E_\infty\eta^t&=s^{-1}\E[\eta^t(1-q^N)].
\end{align*}
Thus the entire tilted walk law is preserved, including its ladder
condition. Since
$s\ind_{\{N\ge1\}}\le1-q^N\le\ind_{\{N\ge1\}}$,
the backbone mark moment is finite exactly under \eqref{C6:raymoment}.
Apply Theorems~\ref{C6:suff} and \ref{C6:nec} to this tree. Under independence,
the mark law is unchanged and remains independent of backbone reproduction.

In the infinite case divergence holds simultaneously in every existing
backbone subtree, by countability of vertex labels. Choose successively a
strict descendant of absolute amplitude at least $1$, then within its
subtree one of amplitude at least $2$, and so on. The finite positive
multiplier from the previously chosen vertex does not change unboundedness.
These nested vertices have depths tending to infinity and define a single
ray with unbounded amplitudes. This proves Corollary~\ref{C6:backbone}.

\subsection{Infinite negative variance is genuinely covered}
Take $t=1$ and the tilted law
\[
 Q(1)=\frac23,\qquad
 Q(-k)=\frac{4}{3k(k+1)(k+2)},\quad k\ge1.
\]
The negative probabilities sum to $1/3$, their first absolute moment is
$2/3$, and their second moment is infinite. Hence $\E S_1=0$,
$\E(S_1^+)^2=2/3$, and $\E S_1^2=\infty$. Realize this law by three
children at $+1$ with probability $2e/9$, one child at $-k$ with
probability $Q(-k)e^{-k}$, and no children otherwise. The probabilities
sum to less than one because $e<3$ and $e^{-k}<1$.
Only the three-child event contributes to $\Omega$, which is finite.
Indeed here $w_i=2/e$ for those three children, so
$\Omega=16/(3e)$. The exact criterion of Theorem~\ref{C6:nec} applies to
independent marks, whereas a finite-variance criterion does not. This example is for the independent-mark theorem:
$\E Z^p=\infty$ for every $p>1$.

\section{Positive drift and logarithmic reproduction moments}\label{positive-proof}
\subsection{The classical additive-martingale input}

The normalized generation sums are
\[
 W_n=\sum_{|u|=n}e^{-tV(u)}.
\]
The classical Biggins martingale theorem, in the form given by Lyons
\cite{Lyons}, implies under \(\E|S_1|<\infty\), \(\mu>0\), and
\eqref{C7:llogl} that
\begin{equation}\label{C7:ui}
 W_n\longrightarrow W_\infty\text{ in }L^1\text{ and almost surely},
 \qquad \E W_\infty=1,\quad
 \{W_\infty>0\}=\Ssurv\ \hbox{a.s.}
\end{equation}
For clarity, set \(a_i=e^{-tY_i}\). The normalized weight conditions are
\[
 \E\sum_i a_i=1,\qquad
 \E\sum_i a_i|\log a_i|=t\E_Q|S_1|<\infty,\qquad
 \E\sum_i a_i\log a_i=-t\mu<0.
\]
Together with \eqref{C7:llogl}, these are the integrability and strict
entropy-sign conditions used for uniform integrability. With a reference
written using positive exponential weights \(e^{\vartheta X_i}\),
the translation is \(X_i=-Y_i\), \(\vartheta=t\), giving
\(\vartheta\E\sum_i X_i e^{\vartheta X_i}=-t\mu<0\)
at normalization one. Thus no ambiguous derivative convention is needed.
In the Galton--Watson specialization, the same classical theorem gives
\begin{equation}\label{C7:kszero}
 \E[N\log^+N]=\infty
 \quad\Longrightarrow\quad m^{-n}Z_n\longrightarrow0\quad\hbox{a.s.},
\end{equation}
where \(Z_n\) denotes population size. We use these standard conclusions,
not a new proof of the martingale convergence theorem.

\section{Stopping lines replace the sibling second moment}
\subsection{The martingale at a first-crossing line}
For each integer \(k\ge0\), let \(\mathcal L_k\) consist of vertices first
crossing strictly above \(k\) on their ancestral path. Write
\[
 M_k=\sum_{u\in\mathcal L_k}e^{-tV(u)}.
\]
We define the exploration sigma-field on the countable Ulam--Harris
label space \(\mathcal U\). Let
\(\mathcal X_v=(\eta_v,(Y_i(v))_{i=1}^{N(v)})\) be the iid local
pairs, with unused copies at labels that do not exist. Define
\[
 J_{v,k}=\ind_{\{v\in\T,\ \max_{w\preceq v}V(w)\le k\}},\qquad
 \mathcal X_{v,k}=
 \begin{cases}
  \mathcal X_v,&J_{v,k}=1,\\
  \dagger,&J_{v,k}=0,
 \end{cases}
\]
where \(\dagger\) is a fixed cemetery symbol, and set
\[
 \mathcal F_{\mathcal L_k}
 =\sigma(J_{v,k},\mathcal X_{v,k}:v\in\mathcal U).
\]
This is exactly exploration of all vertices before the first strict
crossing, including branches that never cross. It records the labels
and positions of the crossing children through their parents' lists,
without revealing a crossing vertex's own pair. In particular it
reveals no local pairs behind the line, including no unused pairs
at nonexistent labels. The line itself may be countably infinite.
The membership test at a vertex uses only its ancestors' pairs.
Consequently the descendant trees rooted at the line are conditionally
independent fresh trees with the original joint local-pair law.

\begin{lemma}[Line convergence]\label{C7:line}
Under \eqref{C7:ui} and positive integrable tilted drift,
\begin{equation}\label{C7:lineconditional}
 M_k=\E[W_\infty\mid\mathcal F_{\mathcal L_k}],
 \qquad M_k\longrightarrow W_\infty\quad\hbox{a.s. and in }L^1.
\end{equation}
\end{lemma}
\begin{proof}
If \(\tau_k=\inf\{n:S_n>k\}\), positive drift gives \(\tau_k<\infty\)
almost surely. Stopping-line many-to-one yields
\(\E M_k=Q(\tau_k<\infty)=1\).
At generation \(n\), split \(W_n\) into descendants of line vertices
already encountered and the residual paths not yet crossing. The residual
has expectation \(Q(\tau_k>n)\to0\). Let \(W_\infty^{(u)}\) be the
independent martingale limit in the subtree rooted at \(u\). If
\(d_j=\E|W_j-W_\infty|\to0\), the \(L^1\) error in replacing the
already encountered subtrees by their limits is bounded by
\[
 \sum_{j=0}^n Q(\tau_k=j)d_{n-j}\longrightarrow0.
\]
The omitted future line vertices have total expected limiting mass
\(Q(\tau_k>n)\). Therefore
\[
 W_\infty=\sum_{u\in\mathcal L_k}e^{-tV(u)}W_\infty^{(u)}
 \quad\hbox{in }L^1\text{ and a.s.}
\]
Conditionally on the line sigma-field just defined, each subtree limit
has mean one, which proves the conditional identity by conditional
Tonelli. These sigma-fields increase with \(k\): the deeper exploration
allows reconstruction of the exploration stopped at any smaller
threshold. For every existing vertex \(v\), its finite ancestral path
has a finite maximum, so \(J_{v,k}=1\) for all sufficiently large \(k\).
Their union therefore generates the whole realized marked tree
(unused pairs at nonexistent labels are irrelevant).
In particular \(W_\infty\) is measurable with respect to this generated
sigma-field. Martingale convergence proves \eqref{C7:lineconditional}.
\end{proof}

\subsection{Uniform overshoot tightness}
Let \(R_k=S_{\tau_k}-k\).
\begin{lemma}[Overshoot truncation]\label{C7:overshoot}
Under \eqref{C7:positive},
\begin{equation}\label{C7:tight}
 \delta(B):=\sup_{k\ge0}Q(R_k>B)\longrightarrow0
 \quad(B\longrightarrow\infty).
\end{equation}
\end{lemma}
\begin{proof}
First we verify \(\E H^+<\infty\) directly. Choose \(D>0\) so that
the capped increment \(\min(S_1,D)\) has mean \(\delta>0\).
This is possible by integrability and \(\mu>0\). If
\(\widetilde T\) is its first strict ascent above zero, the stopped capped
walk is always at most \(D\). Bounded stopping gives
\[
 \delta\,\E(\widetilde T\wedge n)
   =\E\widetilde S_{\widetilde T\wedge n}\le D.
\]
The original walk dominates the capped walk, so
\(\E T^+\le\E\widetilde T\le D/\delta<\infty\).
Wald's identity, justified by
\(\E\sum_{i\le T^+}|S_i-S_{i-1}|=\E T^+\E|S_1|<\infty\),
then gives \(\E H^+=\mu\E T^+<\infty\).

The first crossing above \(k\) occurs at a strict ascending ladder epoch.
Let \(H\) have the law of \(H^+\), and let \(U_H\) be its renewal measure.
For \(B>0\),
\[
 Q(R_k>B)=\int_{[0,k]}U_H(dz)\,Q(H>k+B-z).
\]
Partition \(k-z\) into unit intervals and apply \eqref{C6:renewupper}.
Endpoint atoms can be accommodated by closed intervals; their overlap
only changes a fixed constant. Hence
\[
 \delta(B)\le C\sum_{j\ge0}Q(H>B+j).
\]
The right side tends to zero because \(\E H<\infty\).
\end{proof}

Define the good and bad line masses
\[
 G_k^{(B)}=\sum_{\substack{u\in\mathcal L_k\\V(u)\le k+B}}e^{-tV(u)},
 \qquad D_k^{(B)}=M_k-G_k^{(B)}.
\]
By many-to-one, \(\E D_k^{(B)}=Q(R_k>B)\le\delta(B)\).

\begin{lemma}[Weighted good-line divergence]\label{C7:weighted}
For some fixed integer \(B\ge1\) and \(c_*>0\), every deterministic
sequence \(g_k\in[0,1]\) with \(\sum g_k=\infty\) satisfies
\[
 \PP\left(\sum_{k\ge0}g_kG_k^{(B)}=\infty\right)\ge c_*.
\]
The constants are independent of the sequence.
\end{lemma}
\begin{proof}
Choose \(\varepsilon>0\) with
\(\rho:=\PP(W_\infty\ge4\varepsilon)>0\). Choose \(B\) so large that
\(\delta(B)\le\varepsilon\rho/4\). If \(A_R=\sum_{k=0}^R g_k\to\infty\),
\eqref{C7:lineconditional} and weighted Cesaro convergence give
\[
 A_R^{-1}\sum_{k=0}^R g_k M_k\longrightarrow W_\infty
 \quad\hbox{a.s.}
\]
For sufficiently large \(R\), the probability that this average is at
least \(3\varepsilon\) is at least \(3\rho/4\). Markov's inequality bounds
the probability that the bad-mass average exceeds \(\varepsilon\) by
\(\delta(B)/\varepsilon\le\rho/4\). Thus the good-mass sum exceeds
\(2\varepsilon A_R\) with probability at least \(\rho/2\).
Monotonicity and \(A_R\to\infty\) give the assertion with \(c_*=\rho/2\).
\end{proof}

\subsection{Marks and the survival upgrade}
\begin{proof}[Proof of the necessity assertion in Theorem~\ref{C7:positive-main}]
Fix \(L>0\). The critical tail integral implies
\[
 \sum_{k\ge0}e^{tk}\PP(\eta>L e^{k+B})=\infty.
\]
Set \(g_k=\min\{1,e^{tk}\PP(\eta>L e^{k+B})\}\). If a good line
vertex is at \(v\in(k,k+B]\), then
\[
 e^{-tv}g_k\le\PP(\eta>L e^{k+B})\le\PP(\eta>L e^v).
\]
Each vertex can be good for at most \(B+1\) integer thresholds, regardless
of any overlap between stopping lines. Consequently
\[
 \sum_{k\ge0}g_kG_k^{(B)}
 \le (B+1)\sum_{u\in\T}\PP(\eta>L e^{V(u)}).
\]
Lemma~\ref{C7:weighted} makes the latter compensator infinite with probability
at least \(c_*\), uniformly in \(L\). In a breadth-first exploration,
existence and position are known before each local pair is revealed, so
its marginal mark law is fresh. Conditional Borel--Cantelli gives
infinitely many exceedances on compensator divergence. This remains true
under mark--reproduction dependence.

Thus \(\PP(\sup_u A(u)>L)\ge c_*\) uniformly in \(L\); letting \(L\)
increase gives positive probability of an infinite supremum. Its
finiteness is inherited by every child subtree, because all multipliers,
marks, and offspring lists are finite. Its probability \(r\) solves
\(r=\E r^N\). The supercritical fixed points are \(q\) and \(1\);
positive divergence excludes \(1\). This proves divergence on survival.
\end{proof}

For completeness, conditional Borel--Cantelli can be obtained from the
nonnegative mean-one martingale
\[
 \frac{\exp(-\sum_{j\le n}I_j)}
 {\prod_{j\le n}[1-(1-e^{-1})p_j]},
\]
where \(p_j\) are predictable exceedance probabilities. It would tend to
infinity on \(\{\sum p_j=\infty,\sum I_j<\infty\}\), which has probability
zero by martingale convergence.

The finite direction follows from Theorem~\ref{C6:suff}.
For the infinite-ray assertion, retain only surviving children. Under
conditioning on survival, the tilted intensity is unchanged and
\[
 \E_\infty[W_{\rm back}\log^+W_{\rm back}]
 \le (1-q)^{-1}\E[W\log^+W],\qquad
 \E_\infty Z_{\rm back}^p\le(1-q)^{-1}\E Z^p.
\]
The backbone mark moment is
\(\E[\eta^t(1-q^N)]/(1-q)\), finite exactly when
\(\E[\eta^t\ind_{\{N\ge1\}}]<\infty\). Apply the preceding theorem to
this no-leaf tree. In the divergent case it holds in every surviving
subtree simultaneously. Successively choose nested descendants of
amplitudes at least \(1,2,3,\ldots\); they define a ray with unbounded
amplitudes. These arguments prove all of Theorem~\ref{C7:positive-main}.

\section{The logarithmic condition is a strict weakening}
For a finite list of positive numbers \(a_i\), put \(A=\sum_i a_i\).
Splitting according to whether \(\max_i a_i\ge A/2\) gives
\begin{equation}\label{C7:entropy}
 A\log^+A\le
 2\sum_i a_i\log^+a_i+(\log2)A+2\sum_{i\ne j}a_i a_j.
\end{equation}
Indeed, in the first case use \(A\le2\max a_i\); in the second,
\(\sum_{i\ne j}a_i a_j\ge A^2/2\), and \(A\log A\le A^2\) when
\(A\ge1\). The case \(A<1\) is immediate.
For \(a_i=e^{-tY_i}\),
\[
 \E\sum_i a_i\log^+a_i=t\E S_1^-.
\]
Thus \(\Omega<\infty\) and \(\E|S_1|<\infty\) imply \eqref{C7:llogl}.

To see strictness even with both positive and negative displacements, let
\(H\) have law
\[
 \PP(H=k)=\frac{C}{k^2(\log k)^3},\qquad k\ge2,
\]
where \(C\) normalizes the probabilities, and put \(h=\E H<\infty\).
Let \(N=H\) with probability \(\varepsilon=5/[8(h-1)]\) and \(N=1\)
otherwise. Then \(\E N=13/8\), \(\E[N\log N]<\infty\), but
\(\E N^p=\infty\) for every \(p>1\). Take \(t=1\); independently for
each child, use \(Y_i=\log2\) with probability \(12/13\), and
\(Y_i=-\log2\) with probability \(1/13\).
The exact tilted law is
\[
 Q(\log2)=\frac34,\qquad Q(-\log2)=\frac14.
\]
Its drift is \((\log2)/2>0\), and \(H^+=\log2\) almost surely.
Since \(N/2\le W\le2N\), \eqref{C7:llogl} holds but every positive power
moment of \(W\) above one is infinite. Also \(\Omega=\infty\), since
each \(e^{-Y_i}(1+Y_i^+)\) has a fixed positive lower bound and
\(\E N(N-1)=\infty\). Theorem~\ref{C7:positive-main} therefore gives the
exact independent-mark criterion in a spatial model excluded by the
Theorem~\ref{C6:nec}. There is no extinction because \(N\ge1\).

\section{Why all reproduction assumptions cannot be removed}
\begin{proof}[Proof of Theorem~\ref{C7:universality}]
If \(\E[N\log^+N]<\infty\), the constant-displacement model satisfies
\eqref{C7:positive} and \eqref{C7:llogl}. Theorem~\ref{C7:positive-main} and
Theorem~\ref{C6:suff} give the claimed equivalence, including
infinite rays on survival.

Suppose now that \(\E[N\log^+N]=\infty\). Write \(Z_n\) for the generation
size and \(U_n=Z_n/m^n\). By \eqref{C7:kszero}, \(U_n\to0\) almost surely.
Hence the smoothed populations
\[
 D_\ell=m^{-\ell}\sum_{n=0}^\ell Z_n
       =\sum_{n=0}^\ell m^{n-\ell}U_n
 \longrightarrow0\quad\hbox{a.s.}
\]
To justify the limit, first bound the tail of \(U_n\) uniformly by a small
constant, and then let the finitely many earlier terms decay
geometrically.

Choose strictly increasing deterministic integers \(n_j\ge j\) so large
that
\begin{equation}\label{C7:choose}
 q_j:=\frac{m^{-n_j}}j\le2^{-j-1},\qquad
 \PP\{D_{n_j+j}>j m^{-j}2^{-j}\}\le2^{-j}.
\end{equation}
For every fixed \(j\), convergence in probability of \(D_\ell\) makes
such a choice possible arbitrarily far out. This is an existence
construction depending only on the offspring law, not on a realized tree.
Set
\[
 \PP\{\eta=m^{n_j/t}\}=q_j,\qquad
 \PP\{\eta=1\}=1-\sum_{j\ge1}q_j.
\]
Attach these marks independently of the tree and of one another.
Then \(\eta<\infty\) almost surely and
\[
 \E\eta^t\ge\sum_{j\ge1}q_jm^{n_j}=\sum_{j\ge1}\frac1j=\infty.
\]
For \(0<r<t\), the corresponding sum
\(\sum_j j^{-1}m^{-n_j(1-r/t)}\) is finite because \(n_j\ge j\).
Writing \(a_j=m^{n_j/t}\), for \(a_{j-1}\le x<a_j\) and large \(j\),
strict increase of the \(n_j\) gives
\[
x^t\PP(\eta>x)
 \le m^{n_j}\sum_{k\ge j}\frac{m^{-n_k}}{k}
 \le\frac{1}{j(1-m^{-1})}\longrightarrow0.
\]

Let \(E_j\) be the event that some vertex in generations \(0,\ldots,n_j+j\)
has the \(j\)th large mark. Conditionally on the tree,
\[
 \PP(E_j\mid\T)\le q_j\sum_{n=0}^{n_j+j}Z_n
                   =\frac{m^j}{j}D_{n_j+j}.
\]
Using \eqref{C7:choose}, \(\PP(E_j)\le2^{1-j}\). Borel--Cantelli implies
that only finitely many \(E_j\) occur. For every sufficiently large \(j\),
all occurrences of its mark are therefore after generation \(n_j+j\).
On any path, their total contribution is at most
\[
 m^{n_j/t}\sum_{n>n_j+j}m^{-n/t}
   =\frac{m^{-(j+1)/t}}{1-m^{-1/t}}.
\]
This is summable in \(j\). The finitely many exceptional mark types
have finite values, and their path sums are bounded by geometric series;
the baseline mark one has the same property. Thus \(X_{\rm all}<\infty\)
almost surely.

For uniform genealogical convergence, first discard the sufficiently
large mark types using the summable bound above, uniformly over paths.
The remaining finitely many types are uniformly bounded and have
geometric tails in depth. Compactness gives attainment on the full
boundary and on its closed infinite-ray subset when nonempty.
This proves the counterexample and the converse implication.
\end{proof}

One concrete admissible offspring law is
\(\PP(N=k)=C/[k^2(\log k)^2]\), \(k\ge2\).
It has finite mean greater than one, infinite \(N\log N\) moment,
and no extinction. Thus the obstruction is compatible with finite
expected offspring and a deterministic positive tilted increment with
all moments finite. The marks above are deliberately sparse; the theorem
does not settle necessity for regularly varying marks without
the logarithmic reproduction moment.

\section{Negative drift and weighted path estimates}\label{negative-proof}
The polynomial control of occupation along paths has an antecedent in
the discounted-sum argument of A\"id\'ekon, Hu and Shi \cite{AHS}.
Here we combine band potentials with truncated powers of critical marks.
The resulting path criterion, especially for $t>1$, is the conclusion
emphasized in this section.

An exponential normalization need not be the first positive root.
Write $m(r)=\E\sum_i e^{-rY_i}$.
\begin{proposition}[Negative drift and the first root]\label{first-root}
Suppose $m(t)=1$, $\E_Q S_1^+<\infty$, and $\E_Q S_1<0$, allowing
$-\infty$. If $1<\E N<\infty$, then $t$ is not the first positive root
of $m(r)=1$.
\end{proposition}
\begin{proof}
Since $\E_Q e^{tS_1}=\E N<\infty$, the transform is finite and
continuous on $[0,t]$, and
\[
 m(t-\varepsilon)=\E_Q e^{\varepsilon S_1},\qquad
 0\le\varepsilon\le t.
\]
The right derivative at zero is $\E_Q S_1<0$. For the positive part,
dominated convergence applies using the exponential moment; for the
negative part, monotone convergence permits an infinite negative
derivative. Thus $m(t-\varepsilon)<1$ for some sufficiently small
$\varepsilon>0$. Since $m(0)>1$, continuity gives a strictly smaller
positive root.
\end{proof}
The negative-drift criteria below apply at the stated normalization.
The examples in Appendix~\ref{examples-appendix} have infinite mean
offspring and do realize the first positive root, while each individual
offspring list remains finite.

We first state a general sufficient condition, valid for any $t>0$.
Let $\phi(v)=\E\min(e^{-v}\eta,1)$.
\begin{proposition}[Potential criterion]\label{C6:potentialcriterion}
If, for every nonnegative integer $a$,
\begin{equation}\label{C6:integralcriterion}
 \int_{[-a,\infty)}e^{tv}\phi(v)G_a(0,dv)<\infty,
\end{equation}
then $\sum_{u\in\T}A(u)<\infty$ almost surely. Local dependence is
allowed, and no sibling or weighted reproduction moment is required.
\end{proposition}
\begin{proof}
The mark at a vertex is fresh relative to its existence, position, and
ancestral barrier event. The integral is therefore
$\E\sum_{u\in\T_a}\min(A(u),1)$ by many-to-one. Its finiteness implies
that this nonnegative sum is finite almost surely. Only finitely many
amplitudes exceed one, each is finite, and the others have summable total.
Now localize by \eqref{C6:localize}, with integer barriers.
\end{proof}

\begin{lemma}[Polynomial killed occupation]\label{C6:negmoment}
Suppose $\E S_1^+<\infty$ and $\E S_1<0$, allowing $-\infty$.
For any $r\ge0$ such that $\E(S_1^+)^{r+1}<\infty$, and every $a\ge0$,
\begin{equation}\label{C6:occupationmoment}
 \int_{[-a,\infty)}(1+v^+)^r G_a(0,dv)<\infty.
\end{equation}
\end{lemma}
\begin{proof}
The walk tends to $-\infty$ almost surely. This follows from the strong
law in the integrable case; if the negative mean is infinite, cap negative
increments below $-D$, choose a cap with negative integrable mean, and
dominate the original walk. Its strict ascending ladder law is defective:
if every successive strict ascent were proper, iid positive heights would
force an infinite supremum, a contradiction. Write its mass as $\rho<1$.

For $r>0$, the overshoot decomposition used in Lemma~\ref{C6:lower} applies
to the defective height measure $L$, and gives
\[
 \int z^r L(dz)
 \le C\sum_{m\ge0}\E[(S_1-m)_+^r]
 \le C_r\E\bigl[(S_1^+)^r+(S_1^+)^{r+1}\bigr]<\infty.
\]
For $r=0$ only $\rho<1$ is needed. Normalize $L$ by $\rho$ when
$\rho>0$. For $n$ increments, subadditivity when $r\le1$, or
$(\sum_{i=1}^n z_i)^r\le n^{r-1}\sum_i z_i^r$ when $r>1$, shows that
the $r$th moment of $L^{*n}$ is bounded by a constant times
$n^{\max(1,r)}\rho^{n-1}\int z^rL(dz)$. Summing in $n$ is finite.
Deleting zero weak heights as in \eqref{C6:weakstrict} shows that $U_+$ has
finite total mass and finite $r$th moment. The formula remains valid for
defective strict and weak laws: each zero height contributes a geometric
factor $1/(1-p_0)$.

For a start at $a$, factorization \eqref{C6:factor} integrates only
$h\in[0,a]$, with finite $U_-$ mass, and the endpoint is $a-h+z$.
Its polynomial weight is at most a constant depending on $a,r$ times
$(1+z)^r$. Translate back by $a$ to prove \eqref{C6:occupationmoment}.
\end{proof}

\begin{theorem}[Logarithmic negative-drift bounds]\label{C6:neglog}
Assume the negative-drift hypotheses of Lemma~\ref{C6:negmoment}.
\begin{enumerate}
\item If $0<t<1$, $r\ge0$, $\E(S_1^+)^{r+1}<\infty$, and
\begin{equation}\label{C6:logtail}
 \PP(\eta>x)\le Cx^{-t}(1+\log x)^r\quad(x\ge1),
\end{equation}
then $\sum_{u\in\T}A(u)<\infty$ almost surely.
\item If $t=1$, $r\ge0$, $\E(S_1^+)^{r+2}<\infty$, and
\eqref{C6:logtail} holds with $t=1$, the same conclusion holds.
\end{enumerate}
These conclusions allow real displacements and local dependence, without
an assumption on $\Omega$ or on $Z^p$.
\end{theorem}
\begin{proof}
For $v\ge0$, the layer-cake formula gives
\[
 \phi(v)=e^{-v}\int_0^{e^v}\PP(\eta>x)\,dx.
\]
If $0<t<1$, \eqref{C6:logtail} bounds this by
$C'e^{-tv}(1+v)^r$: bound the logarithmic factor by $(1+v)^r$ and
integrate $x^{-t}$ from $1$ to $e^v$. If $t=1$, integrating
$x^{-1}(1+\log x)^r$ gives
$\phi(v)\le C'e^{-v}(1+v)^{r+1}$. The bounded interval $[-a,0]$
causes no difficulty by Lemma~\ref{C6:potential}. Lemma~\ref{C6:negmoment}
and Proposition~\ref{C6:potentialcriterion} now apply with occupation
powers $r$ and $r+1$, respectively.
\end{proof}

\section{Negative drift: a linear band bound}\label{C7:negative}
Assume henceforth \(\E S_1^+<\infty\) and \(\E S_1<0\), possibly
\(-\infty\). The walk tends to \(-\infty\) almost surely. For an infinite
negative mean, cap negative increments at \(-D\), choose the cap with
negative integrable mean, and dominate the original walk.

\begin{lemma}[Full potential and linear ray local time]\label{C7:linear}
The full-walk occupation measure satisfies, for each fixed \(b>0\),
\[
 \sup_{x,y\in\mathbb R}
 \E_x\sum_{n\ge0}\ind_{\{S_n\in[y,y+b)\}}<\infty.
\]
Consequently, on every tree killed below \(-a\), there is a finite random
constant \(C_a\) such that every path satisfies
\begin{equation}\label{C7:linearcount}
 \#\{u\in\pi\cap\T_a:V(u)\in[k,k+1)\}\le C_a(k+a+2),
 \qquad k\ge-a,
\end{equation}
with integer \(a\ge0\). In particular, for every \(q_0>2\),
\begin{equation}\label{C7:pathweight}
 \sup_\pi\sum_{u\in\pi\cap\T_a}(1+V(u)^+)^{-q_0}<\infty
 \quad\hbox{a.s.}
\end{equation}
\end{lemma}
\begin{proof}
The strict ascending ladder law is defective: if every ascent were
proper, its iid positive heights would force an infinite supremum.
The weak ascending law also has mass less than one after zero-height
renewals are accounted for. Hence \(U_+([0,\infty))<\infty\).
Split an unrestricted finite path at its first global minimum. Unlike
the killed formula, its descending drop \(h\) is now unrestricted.
The full occupation measure \(R\) therefore satisfies
\[
 R(0,B)=\int_{[0,\infty)}U_+(dz)\,
          U_-(\{h\ge0:z-h\in B\}).
\]
The inner measure is uniformly bounded for intervals of length \(b\)
by \eqref{C6:renewupper}. Integrating the finite mass of \(U_+\) proves
the full potential bound, including all starting points by translation.

Thus the expected number of visits to a unit band from any point of
that band, before killing, is bounded by a constant independent of the
band's height. Markov's inequality in blocks of visits and the strong
Markov property give
\(\PP\{L([k,k+1))\ge r\}\le C e^{-cr}\), uniformly in the band.
Many-to-one for the line of \(r\)th band visits bounds its expected
population by \(Ce^{t(k+1)-cr}\). Choose \(r\) a sufficiently large
multiple of \(k+a+2\) and use Borel--Cantelli. Every band has finite
population by \eqref{C6:manyone}, so finitely many exceptions are absorbed in
\(C_a\). This proves \eqref{C7:linearcount}. Summing
\((k+a+2)(1+k)^{-q_0}\) proves \eqref{C7:pathweight}.
\end{proof}

By Lemma~\ref{C6:negmoment}, for $d\ge0$,
\begin{equation}\label{C7:polymoment}
 \E(S_1^+)^{d+1}<\infty\quad\Longrightarrow\quad
 \int(1+v^+)^dG_a(0,dv)<\infty.
\end{equation}

\section{Weighted powers control path sums at every positive exponent}
\begin{theorem}[Weighted-power criterion]\label{C7:power}
Assume negative drift as in Section~\ref{C7:negative} and the mark-tail bound
\[
 \PP(\eta>x)\le Cx^{-t}(1+\log x)^{-\beta}\quad(x\ge1).
\]
Suppose there exist \(p>\max\{1,t\}\) and \(\gamma>2(1-1/p)\) such that
\begin{equation}\label{C7:powercondition}
 \E(S_1^+)^{1+\max\{0,\gamma t-\beta\}}<\infty.
\end{equation}
Then, almost surely,
\begin{equation}\label{C7:weightedtotal}
 \sum_{u\in\T} A(u)^p(1+V(u)^+)^{p\gamma}<\infty.
\end{equation}
It follows that all path sums converge uniformly and have an attained
finite maximum, including on infinite rays on survival. Arbitrary local
mark--reproduction dependence is allowed.
\end{theorem}
\begin{proof}
For \(p>t\), the tail integral implies, for all sufficiently large \(z\),
\begin{equation}\label{C7:powertrunc}
 \E\min\{(\eta/z)^p,1\}
 =p z^{-p}\int_0^z x^{p-1}\PP(\eta>x)\,dx
 \le C' z^{-t}(1+\log z)^{-\beta}.
\end{equation}
To verify the bound for either sign of \(\beta\), substitute \(u=\log x\)
in the integral above one. The resulting exponential factor has positive
rate \(p-t\). On the last half of the interval the polynomial factor is
comparable to its endpoint value; the earlier half is exponentially
smaller and can be absorbed. The bounded initial integral is negligible.

Set \(z=e^v/(1+v)^\gamma\) for \(v\ge0\).
For large \(v\), \(\log z\) is comparable to \(v\), so
\[
 \E\min\{(e^{-v}\eta)^p(1+v)^{p\gamma},1\}
 \le C e^{-tv}(1+v)^{\gamma t-\beta}.
\]
The mark at a vertex is fresh relative to its ancestral position and
barrier membership. Consequently \eqref{C6:manyone}, \eqref{C7:polymoment}, and
\eqref{C7:powercondition} make the expected sum of these truncated weighted
powers on \(\T_a\) finite. Negative powers of \(1+v\) can be bounded
by one; bounded spatial intervals have finite potential. Only finitely
many weighted powers exceed one, and each is finite. Their untruncated
sum is thus finite. Localizing in \(a\) proves \eqref{C7:weightedtotal}.

Put \(q_0=\gamma p/(p-1)>2\). For every path and depth cutoff \(n\),
H\"older gives
\begin{align*}
 \sum_{\substack{u\in\pi\\|u|\ge n}}A(u)
 &\le
 \left(\sum_{\substack{u\in\T\\|u|\ge n}}
        A(u)^p(1+V(u)^+)^{p\gamma}\right)^{1/p}\\
 &\quad\times
 \left(\sum_{u\in\pi}(1+V(u)^+)^{-q_0}\right)^{(p-1)/p}.
\end{align*}
The first factor tends to zero by \eqref{C7:weightedtotal}; after localization,
the second is bounded uniformly over paths by \eqref{C7:pathweight}.
This gives uniform convergence. Compactness then gives attainment.
\end{proof}

\begin{proof}[Proof of Corollary~\ref{C7:negative-main}]
As \(p\downarrow\max\{1,t\}\), the lower bound
\(2t(1-1/p)\) tends to \(2(t-1)_+\).
Under \eqref{C7:simple-negative}, choose \(p>\max\{1,t\}\) and then
\(\gamma>2(1-1/p)\) so close to these lower limits that
\(\gamma t-\beta<s-1\). Since \(s>1\), the power in
\eqref{C7:powercondition} is strictly below \(s\).
Theorem~\ref{C7:power} applies.
\end{proof}

\section{Scope and open questions}\label{discussion}
The occupation result supplies a positive persistent quadratic lower
bound throughout the centered skip-free class. The forward construction
now optimizes over all compactly supported $H_0^1$ profiles. It gives
$4t/(\sigma^2j_{1,1}^2)\le L\le t/\sigma^2$ for
$\{-1,0,1\}$ displacements and realizes the
lower bound on one ray. The energy quotient is solved exactly, but its
relation to the global upper bound remains unresolved. Extending the
profile-burst argument to general centered finite-variance skip-free
increments is a further question; it is not implied by the present
nearest-neighbor Dirichlet-form calculation. Neither determinism of $L$ nor the existence of maximizing
parameters for a particular construction proves that an arbitrary
globally maximizing ray exists. Appendix~\ref{attainment-appendix}
isolates that distinction.

For marks, finite critical moments, a ladder-based converse, and a
positive-drift $L\log L$ converse give exact criteria in the stated
classes. The constant-displacement obstruction prevents removal of all
reproduction integrability assumptions from a universal converse.
Negative-drift examples show that no criterion based only on the marginal
critical mark moment can cover every normalized reproduction law.

Several questions remain. A centered converse under weaker reproduction
conditions than $\Omega<\infty$ would complement the positive-drift
argument. The infinite-mean ascending-height case requires another
mechanism. At negative drift, an exact necessary and sufficient criterion
and the boundary cases in the rational example are unresolved. The sparse
universality obstruction does not decide regularly varying marks when
$L\log L$ fails. General local dependence, uniqueness of recursive
solutions, and their tail asymptotics also remain outside these results.

\clearpage
\appendix
\section*{Appendices}
The appendices give the skip-free excursion normalization (A), an exact
certificate for the rational occupation bound (B), the distinction
between constructional and global attainment (C), and explicit
negative-drift examples (D).

\section{Normalization of the skip-free excursion constant}\label{excursion-appendix}

\begin{proof}[Proof of Theorem \ref{A4:theta}]
Put $X=1-S_1\in\mathbb Z_+$, $p_j=\PP(X=j)$, and $f(s)=\E s^X$. Then $\E X=1$, $\operatorname{Var}X=\sigma^2$, and $p_0>0$. Define a probability law
\[
 g_j=\PP(X>j),\quad j\ge0,
 \qquad \sum_jg_j=1,\quad
 \mu:=\sum_jjg_j=\frac{\E X(X-1)}2=\frac{\sigma^2}2.
\]
Let $Z_1,Z_2,\ldots$ have law $g$, let $T_n=Z_1+\cdots+Z_n$, and set
\[
 u_j=\sum_{n\ge0}\PP(T_n=j),\qquad U(m)=\sum_{j=0}^m u_j.
\]
Zero increments are allowed; $g_0=1-p_0<1$, so these quantities are finite. The generating-function identities are
\[
 f(s)-s=(1-s)(1-G(s)),\qquad
 \sum_{n\ge0}w_ns^n=\frac{p_0s}{f(s)-s},
 \quad w_0=0,\quad w_n=p_0U(n-1)\ (n\ge1).
\]
In particular $w_1=1$. Extend $w_n=0$ for $n\le0$. Coefficient comparison gives
\[
 w_n=\sum_{j\ge0}p_jw_{n+1-j},\qquad n\ge1.
\]
Thus $w$ is harmonic for the walk killed at its first entry into $(-\infty,0]$. It is nondecreasing by its renewal representation. Stopping its bounded harmonic martingale upon reaching $h$ or $(-\infty,0]$ gives
\[
 \PP_1(T_h<T_{(-\infty,0]})=1/w_h.
\]
This stopping time is finite almost surely by recurrence; the upward endpoint is exactly $h$ because the walk is skip-free upward.

A first-return excursion can reach $h>0$ only if its first step is $+1$, and then it must reach $h$ before entering $(-\infty,0]$. Once negative, it must visit zero before it can become positive again. Its maximum therefore exceeds or equals $h$ with probability $p_0/w_h=1/U(h-1)$. Reversing the order of a finite first-return increment sequence preserves its probability and changes its partial sums to the negatives of the old partial sums in reverse time. Hence maximum and negative minimum have the same distribution, and
\begin{equation}\label{A4:exactq}
 q(h)=\frac1{U(h-1)}.
\end{equation}

For completeness, the needed elementary renewal limit follows directly from truncation. Let $\nu_m=\inf\{n\ge1:T_n>m\}$, so $\E\nu_m=U(m)$. For increments truncated at an integer $a$, write $\mu_a=\E(Z_1\wedge a)$. The corresponding crossing time dominates $\nu_m$ and has finite mean, since positive increments have positive probability. Its overshoot is at most $a$. Wald's identity, proved here by summing the independent terms $Z_n\ind_{\{\nu_m\ge n\}}$, gives
\[
 \frac m\mu\le U(m)\le\frac{m+a}{\mu_a}.
\]
Let $m\to\infty$ and then $a\to\infty$. Thus $U(m)/m\to1/\mu$. Equation \eqref{A4:exactq} yields $hq(h)\to\mu=\sigma^2/2$. The asserted bounds follow from AHS and Theorem~\ref{A3:general}.
\end{proof}

\section{Exact certificate for \texorpdfstring{$17/160$}{17/160}}\label{A4:certificate}
For $B=9$ put
\[
 I=\int_1^{10}\frac{\sin^2(\pi(r-1)/9)}{r^3}\,dr,
 \qquad\mathcal C(9)=18I/\pi^2.
\]
The following finite rational certificate gives a rigorous lower bound, not a floating-point quadrature assumption. Let
\[
 N=10000,\quad D=10^{10},\quad E=10^{15},\quad
 p_-=314159/100000,\quad p_+=3927/1250.
\]
The bounds $p_-<\pi<p_+$ follow from Machin's identity
\[\pi=16\arctan(1/5)-4\arctan(1/239)\]
and alternating arctangent sums. Taking twelve and thirteen terms for $1/5$, and four and five terms for $1/239$, already verifies these rational inequalities. The identity itself follows from the tangent addition formula with its angle in $(0,\pi/2)$.

For $0\le j\le N$ define
\[
 y_j=p_-\min(j/N,1-j/N),\qquad
 s_j=\left\lfloor D\sum_{m=0}^{5}
             \frac{(-1)^m y_j^{2m+1}}{(2m+1)!}\right\rfloor.
\]
All $s_j$ are nonnegative. The alternating sine remainder and monotonicity on $[0,\pi/2]$ give $s_j/D\le\sin(\pi j/N)$. On each interval $[j/N,(j+1)/N]$, sine concavity bounds it below by the smaller endpoint value. Also $r^{-3}$ is decreasing. Hence
\[
 I\ge\frac1E\sum_{j=0}^{N-1}
 \left\lfloor
 \frac{9N^2E\,\min(s_j,s_{j+1})^2}
      {D^2\{N+9(j+1)\}^3}
 \right\rfloor
 =\frac{58294384942850}{10^{15}}.
\]
The integer equality is reproduced by the supplied exact-arithmetic script. Therefore
\[
 \mathcal C(9)\ge
 \frac{18}{p_+^2}\frac{58294384942850}{10^{15}}
 =\frac{166555385551}{1566611200000}
 >\frac{17}{160}.
\]

\section{Global optimal attainment and a compactness obstruction}\label{attainment-appendix}

Theorem \ref{A4:main} constructs a ray at the limiting value of an explicit family, after both the subcritical margin and the relative block size vanish. This is stronger than taking a supremum over separately constructed rays. It still does not identify the resulting bound with $L$.

\begin{proposition}[A topological obstruction]\label{A4:obstruction}
There is a deterministic locally finite tree with nearest-neighbor position labels such that: its generation minimum tends to infinity; every infinite descendant subtree has persistent supremum one; and no ray attains that value. Thus these properties alone, even with compact ray space and continuous finite-level local times, do not imply optimal attainment.
\end{proposition}
\begin{proof}
Use stages $n\ge1$, starting at spatial level $2n$, and an integer mode $q\ge1$. Put $a_q=1-1/(q+1)$. At each stage fork into two finite paths, both beginning with a $+1$ step. The continue path oscillates between $2n$ and $2n+1$ exactly $m=\lfloor a_q(2n)^2\rfloor$ times, then advances to $2n+2$, keeping mode $q$. The switch path makes two positive steps to $2n+2$ and changes the mode to $q+1$. Identify neither path with the other; each endpoint starts its own next-stage subtree. All intermediate vertices have one child. An initial finite path reaches level two.

A ray with finitely many switches has eventual mode $q$ and normalized local-time limit $a_q<1$. A ray with infinitely many switches has infinitely many two-level stages visited only once, so its liminf is zero. From any finite prefix one can make any desired finite number of additional switches and then continue forever. Hence every subtree supremum is one, but it is never attained. Stage durations are at most a constant times $n^2$, uniformly over modes, so the generation minimum tends to infinity. Every fixed level occurs in finitely many vertices, making its local time continuous on the compact ray space.
\end{proof}

\begin{remark}
This is not a counterexample under \eqref{A3:H}. Its reproduction depends on stage and mode and is not iid. It only rules out the proposed purely topological shortcut. Random reproduction structure must do additional work in any proof of attainment for the original model.
\end{remark}

Here is one precise sufficient condition for the missing compactness step. AHS notes that the generation minimum tends to infinity under the present hypotheses. Consequently, in a fixed surviving realization, $f_k(\xi)=N_\xi^k/k^2$ is continuous on the compact ray space: each spatial level appears at finitely many vertices.

\begin{proposition}[A uniform-tail criterion]\label{A4:criterion}
In such a realization, suppose probability measures $\nu_m$ on rays satisfy, for every $\varepsilon>0$,
\begin{equation}\label{A4:uniform}
 \lim_{K\to\infty}\sup_m
 \nu_m\{\exists k\ge K:f_k<L-\varepsilon\}=0.
\end{equation}
Then a ray attains $L$.
\end{proposition}
\begin{proof}
Compactness gives a weakly convergent subsequence, with limit $\nu$. The set in braces is open, being a countable union of inverse images of open intervals under continuous $f_k$. The open-set inequality for weak convergence gives its $\nu$ probability at most the liminf of the subsequence probabilities. By \eqref{A4:uniform}, the probability of infinitely many violations of $f_k\ge L-\varepsilon$ is zero. Intersect over positive rational $\varepsilon$. Then $\nu$-almost every ray has liminf at least $L$, which must equal $L$ by its definition.
\end{proof}

The existence of a ray with value at least $L-1/m$ for every $m$ does not give \eqref{A4:uniform}: the onset of the eventual guarantee can escape to infinity. Following one such ray to height $h$ and switching to another introduces the further translation factor $(k-h)^2/k^2$. The present work supplies the needed uniformity for the explicit forward-profile family, but has not supplied it for arbitrary near-optimal rays.

\section{Explicit negative-drift models and method boundaries}\label{examples-appendix}
The first example has a computable killed potential. The stated boundary
is the boundary of a sufficient first-moment test, not an almost-sure
phase transition.

For any $0<t\le1$, take the tilted law
\begin{equation}\label{C6:examplelaw}
 Q(-1)=d_0:=\frac{17}{18},\qquad
 Q(k)=p_k:=\frac1{k(k+1)(k+2)(k+3)},\quad k\ge1.
\end{equation}
Its positive mass is $1/18$, its positive first moment $1/12$, and its
positive second moment $1/4$. Thus
\[
 \E S_1=-\frac{31}{36},\qquad \E S_1^2=\frac{43}{36},
 \qquad \E(S_1^+)^u<\infty\ \Longleftrightarrow\ u<3
 \quad(u>0).
\]
Set $n_k=\lceil e^{tk}\rceil$. Use one child at $-1$ with probability
$d_0e^{-t}$, $n_k$ children all at $+k$ with probability
$b_k=p_ke^{tk}/n_k$, and no children otherwise. These probabilities sum
to less than one, and the tilted law is exactly \eqref{C6:examplelaw}.
The offspring number is finite almost surely but has infinite mean.
Its transform is infinite for $0<r<t$, is one at $r=t$, and is greater
than one for $r>t$. The last statement follows from the positive right
derivative at $t$, $-\E S_1=31/36$, and convexity; all terms are finite
for $r\ge t$. Hence $t$ is the AHS (H1) minimal positive root.
Also $\E Z^2<\infty$, since
\[
 b_kn_k^2e^{-2tk}(1+k)^2
 =p_kn_ke^{-tk}(1+k)^2\le2p_k(1+k)^2,
\]
and the negative event contributes only $d_0e^t$.

\begin{corollary}[Stronger counterexample family]\label{C6:example}
For this reproduction law, suppose, with arbitrary allowed local
dependence,
\[
 \PP(\eta>x)\asymp x^{-t}(\log x)^{-\beta}
 \quad\hbox{as }x\longrightarrow\infty.
\]
Then the total sum over all vertices is finite almost surely whenever
\[
 \begin{cases}
 \beta>-2,&0<t<1,\\
 \beta>-1,&t=1.
 \end{cases}
\]
The ranges $\beta\le1$ have infinite critical mark moment. Thus the
negative-drift failure of a universal critical-moment converse includes
logarithmically heavier marks and the endpoint $t=1$.
\end{corollary}
\begin{proof}
For $\beta<0$ use $r=-\beta$ in Theorem~\ref{C6:neglog}; the necessary
positive moments are finite exactly for $r<2$ or $r<1$, respectively.
For $\beta\ge0$, bound the tail by $Cx^{-t}$ and use $r=0$.
The critical moment diverges for $\beta\le1$ by its tail integral.
\end{proof}

The following calculation identifies the precise limit of the potential
test, and avoids mistaking that limit for a proved almost-sure transition.
Let $g_k=G(0,\{k\})$ for the walk \eqref{C6:examplelaw}. Since downward
steps are $-1$, a positive excursion from zero returns to zero almost
surely before killing. Consequently $g_0=1/d_0$. The balance equations
and summation of their generating series yield
\begin{equation}\label{C6:greenseries}
 \sum_{k\ge0}g_kz^k
 =\frac{1}{d_0}\frac1{1-\sum_{l\ge1}a_lz^l},\qquad
 a_l=\frac{\PP(S_1\ge l)}{d_0}
     =\frac{6}{17l(l+1)(l+2)},\quad 0\le z<1.
\end{equation}
Indeed the unreduced denominator is
$d_0-z+z\sum_{k\ge1}p_kz^k$; dividing it by $1-z$ gives
$d_0-\sum_{l\ge1}\PP(S_1\ge l)z^l$.
Here $\rho:=\sum_l a_l=3/34<1$, and
$\sum_k g_k=36/31$.

The coefficient of the first convolution in \eqref{C6:greenseries} gives
$g_k\ge c(k+1)^{-3}$. For the upper bound, in every $n$-tuple of positive
integers with sum $k$, some entry is at least $k/n$. Select that entry,
bound its mass by $C n^3/(k+1)^3$, and sum the remaining masses. It follows
that
\[
 (a^{*n})_k\le\frac{C n^4\rho^{n-1}}{(k+1)^3},
 \qquad
 c(k+1)^{-3}\le g_k\le C'(k+1)^{-3}.\label{C6:greenpower}
\]
The $k=0$ term is absorbed separately. For an integer barrier $a$, the
strict descending ladder heights are deterministically one, so
\eqref{C6:factor} gives the exact identity
\[
 G_a(0,\{k\})=\sum_{h=0}^a g_{k+h}\quad(k\ge0).
\]
Thus its positive tail has the same order $k^{-3}$.

For $0<t<1$, elementary integration gives
$e^{tk}\phi(k)\asymp(1+k)^{-\beta}$. For $t=1$ it gives, as
$k\to\infty$,
\[
 e^k\phi(k)\asymp
 \begin{cases}
 k^{1-\beta},&\beta<1,\\
 \log k,&\beta=1,\\
 1,&\beta>1.
 \end{cases}
\]
Consequently \eqref{C6:integralcriterion} is finite exactly in the respective
ranges $\beta>-2$ and $\beta>-1$. At or below these boundaries,
\emph{the expected truncated total is infinite}; this does not imply
almost-sure divergence of the total or of the maximal path sum. The
offspring bursts make that distinction substantial. No converse is claimed.

\subsection{One model accepts every polynomial logarithmic correction}
The preceding rational law is useful because its potential can be
computed. A second law gives a stronger qualitative obstruction. Put
\[
 B=\sum_{k\ge1}(k+1)e^{-\sqrt{k}},\qquad
 c=\frac1{2B},\qquad p_k=c e^{-\sqrt{k}},\qquad
 d_0=1-\sum_{k\ge1}p_k.
\]
Then $Q(-1)=d_0$, $Q(k)=p_k$ defines a walk with mean $-1/2$ and all
positive increment moments finite. For each fixed $0<t\le1$, realize
this tilted law by the same construction: one child at $-1$ with
probability $d_0e^{-t}$, $\lceil e^{tk}\rceil$ children at $+k$ with
probability $p_ke^{tk}/\lceil e^{tk}\rceil$, and otherwise no children.
These probabilities are valid. The transform is infinite below $t$, is
one at $t$, and increases above $t$ by convexity and its positive right
derivative. This is another negative-drift AHS (H1) model. Moreover
\begin{equation}\label{C6:allmoments}
 \E Z^p<\infty\quad\hbox{for every finite }p>1,
 \qquad \Omega<\infty.
\end{equation}
Indeed on the $+k$ event, $Z\le2(1+k)$ and its probability is at most
$p_k$; the negative event has constant $Z=e^t$.

\begin{corollary}[No universal logarithmic threshold]\label{C6:alllogs}
For each fixed $0<t\le1$, this single reproduction model satisfies
\[
 \sum_{u\in\T}e^{-V(u)}\eta_u<\infty\quad\hbox{almost surely}
\]
for every finite $r\ge0$ and every mark law bounded in tail by
$Cx^{-t}(1+\log x)^r$. In particular it allows every fixed real
$\beta$ in the tail $x^{-t}(\log x)^{-\beta}$.
For the same $t$ and any such tail with $\beta\le1$, there is also a
bounded-offspring centered model for which one infinite ray has unbounded
individual amplitudes almost surely. Both constructions may use marks
independent of reproduction.
\end{corollary}
\begin{proof}
The finite assertion follows from Theorem~\ref{C6:neglog}, since every
required positive increment moment is finite. For the centered model,
let $m=\cosh t\in(1,2)$, take $N=2$ with probability $m-1$ and $N=1$
otherwise, and give each child a displacement independently equal to
$+1$ with probability $e^t/(2m)$ and $-1$ with probability $e^{-t}/(2m)$.
Its tilted walk is simple symmetric, its offspring moments are bounded,
and it has no extinction. Since $\E\eta^t=\infty$ for $\beta\le1$,
Corollary~\ref{C6:backbone} gives a ray with unbounded amplitudes.
\end{proof}
The universal quantifier concerns separate choices of mark law for
one fixed reproduction law; it is not a claim about all possible mark
assignments on a common probability space. This stronger obstruction does
not require poor weighted reproduction moments: \eqref{C6:allmoments} holds.

\subsection{Consequences for the explicit examples}
For the rational tilted law
\[
 Q(-1)=\frac{17}{18},\qquad
 Q(k)=\frac1{k(k+1)(k+2)(k+3)},\quad k\ge1,
\]
the positive moment of order \(s\) is finite exactly when \(s<3\).
Thus the new sufficient path-sum range is
\begin{equation}\label{C7:rationalrange}
 \beta>2(t-1)_+-2.
\end{equation}
For \(t=1\), this extends path boundedness to \(\beta>-2\), whereas
the total-tree expectation argument of Theorem~\ref{C6:neglog} gave \(\beta>-1\).
For \(1<t<2\), even the weak Pareto tail \(\beta=0\) is now covered.
The reproduction realization remains valid for all \(t>0\): use one
child at \(-1\) with probability \((17/18)e^{-t}\), and
\(\lceil e^{tk}\rceil\) children at \(+k\) with probability
\(Q(k)e^{tk}/\lceil e^{tk}\rceil\), and otherwise no children.
No divergence claim is made outside \eqref{C7:rationalrange}.

More strongly, let
\[
 B=\sum_{k\ge1}(k+1)e^{-\sqrt{k}},\qquad
 Q(k)=\frac{e^{-\sqrt{k}}}{2B}\ (k\ge1),\qquad
 Q(-1)=1-\sum_{k\ge1}Q(k).
\]
This walk has mean \(-1/2\) and every positive increment moment finite.
For each fixed \(t>0\), realize it by the preceding offspring construction.
It has finite offspring lists, infinite expected offspring, all finite
moments of \(Z\), and \(\Omega<\infty\). Its transform is infinite below
\(t\), equal to one at \(t\), and greater than one above \(t\);
it is a negative-drift AHS (H1) model.

\begin{corollary}[Nonuniversality for every exponent]\label{C7:all-exponents}
For every fixed \(t>0\), this single reproduction law has uniformly
convergent bounded path sums for every fixed real \(\beta\) and every
mark law with tail bounded by \(Cx^{-t}(1+\log x)^{-\beta}\).
For the same \(t\) and a mark tail
\(x^{-t}(\log x)^{-\beta}\) with \(\beta\le1\), a centered
bounded-offspring model has a ray of unbounded amplitudes almost surely.
\end{corollary}
\begin{proof}
Choose \(s\) arbitrarily large in Corollary~\ref{C7:negative-main}.
For the centered comparison let \(m=\cosh t>1\), choose
\(N\in\{\lfloor m\rfloor,\lceil m\rceil\}\) with mean \(m\), and let
each child's displacement be \(+1\) with probability \(e^t/(2m)\)
and \(-1\) with probability \(e^{-t}/(2m)\). If \(m\) is an integer
take \(N=m\). This model has no extinction and tilted probabilities
\(Q(1)=Q(-1)=1/2\). Its weighted sibling moment is finite. The centered
necessity theorem, Theorem~\ref{C6:nec}, applies because the critical mark
moment is infinite for \(\beta\le1\), and supplies the asserted ray.
\end{proof}
For \(t\le1\), the stronger total-tree finiteness conclusion of Corollary~\ref{C6:alllogs}
also holds in this stretched-exponential example. For \(t>1\), the new
argument proves path-sum finiteness and the weighted \(p\)-sum
\eqref{C7:weightedtotal}; it does not assert finiteness of the unpowered
sum over the entire tree.

\begin{acks}[Use of generative AI]
A. Skiredj used generative AI tools to identify and compare related work,
improve the writing and presentation, explore possible results and proof
strategies, and examine draft arguments for errors. The authors take
responsibility for the sources, mathematical claims, proofs, and final text;
AI-assisted checks are not an independent mathematical certification.
\end{acks}

\end{document}